\documentclass[a4paper]{article}
\usepackage[english]{babel}
\usepackage{amsmath, amsfonts, mathtools, mathrsfs, dsfont, empheq, esint} 
\usepackage{graphicx} 

\usepackage[shortlabels]{enumitem} 

\usepackage[colorlinks = true, linkcolor = red, citecolor = blue, final]{hyperref}
\usepackage{amsthm}
\usepackage{thmtools}

\usepackage{nameref}
\usepackage{cleveref}

\declaretheorem[
name=Theorem,
Refname={Theorem,Theorems},
numberwithin=section
]{theorem}

\declaretheorem[
name=Proposition,
Refname={Proposition,Propositions},
sibling=theorem
]{proposition}

\declaretheorem[
name=Lemma,
Refname={Lemma,Lemmas},
sibling=theorem
]{lemma}

\declaretheorem[
name=Corollary,
Refname={Corollary,Corollaries},
sibling=theorem
]{corollary}

\declaretheorem[
name=Claim,
unnumbered
]{claim*}

\declaretheorem[
name=Definition,
style=definition,
Refname={Definition,Definitions},
sibling=theorem
]{definition}

\declaretheorem[
name=Remark,
style=remark,
Refname={Remark,Remarks},
sibling=theorem
]{remark}

\declaretheorem[
name=Remark,
style=remark,
unnumbered
]{remark*}

\Crefname{theorem}{Theorem}{Theorems}
\Crefname{lemma}{Lemma}{Lemmas}

\newcommand{\IN}{{\mathbb{N}}}

\newcommand{\IR}{{\mathbb{R}}}

\newcommand{\vertiii}[1]{{\left\vert\kern-0.25ex\left\vert\kern-0.25ex\left\vert #1 
		\right\vert\kern-0.25ex\right\vert\kern-0.25ex\right\vert}}

\newcommand{\abs}[1]{{\left|{#1}\right|}}
\newcommand{\norm}[1]{{\|#1\|}}
\newcommand{\Norm}[1]{{\vert\kern-0.25ex\vert\kern-0.25ex\vert #1 
	\vert\kern-0.25ex\vert\kern-0.25ex\vert}}
\newcommand{\inprod}[2]{{\left\langle#1,#2\right\rangle}}
\newcommand{\pinprod}[2]{{\left(#1,#2\right)}}

\newcommand{\md}{\mathrm{d}}
\newcommand{\loc}{{\mathrm{loc}}}
\DeclareMathOperator*{\esssup}{ess\,sup}
\DeclareMathOperator*{\essinf}{ess\,inf}

\DeclareMathOperator*{\essosc}{ess\,osc}
\newcommand{\emptyarg}{{\,\cdot\,}}

\newcommand{\solA}{u}	\newcommand{\solAvar}{{\tilde{\solA}}}
\newcommand{\solB}{v}	
\newcommand{\solC}{w}

\newcommand{\testA}{\eta}

\newcommand{\nonlin}{\phi}

\newcommand{\source}{f}

\newcommand{\LipBoundSource}{L}

\newcommand{\Qrho}{{Q^{\nu_0}_\omega(R)}}

\title{Interior H\"older continuity for singular-degenerate porous medium type equations with an application to a biofilm model}
\author{Hissink Muller, Victor\\
Radboud Universiteit, IMAPP - Mathematics\\
PO Box 9010, 6500 GL Nijmegen,
The Netherlands\\
\texttt{V.HissinkMuller@math.ru.nl}}
\date{ }

\begin{document}
	\maketitle
	
\begin{abstract}
	We show interior H\"older continuity for a class of quasi-linear degenerate reaction-diffusion equations.
	The diffusion coefficient in the equation has a porous medium type degeneracy and its primitive has a singularity.
	The reaction term is locally bounded except in zero.
	The class of equations we analyse is motivated by a model that describes the growth of biofilms.
	Our method is based on the original proof of interior H\"older continuity for the porous medium equation.
	We do not restrict ourselves to solutions that are limits in the weak topology of a sequence of approximate continuous solutions of regularized problems, which is a common assumption.
\end{abstract}



\section*{Introduction}
We study the regularity of local solutions of equations of the form
\begin{equation}\label{eq:SD-PME}
	\solA_t=\Delta \nonlin(\solA)+\source(\emptyarg,\solA)\quad\text{in }\Omega\times(0,T],
\end{equation}
where $\Omega\subseteq\IR^N$ is open, $0<T\leq\infty$ and $\solA$ takes values in $[0,1)$.
The key assumptions on $\nonlin$ are that it possesses a singularity $\nonlin(1)=\infty$ and a degeneracy $\nonlin'(0)=0$.
The degeneracy in $0$ is of the same type as observed in the porous medium equation $\solA_t=\Delta\solA^m$.
Therefore, we call \eqref{eq:SD-PME} a \emph{singular-degenerate equation of porous medium type}.
Our main result is that solutions of \eqref{eq:SD-PME} that are bounded away from $1$ are H\"older continuous in the interior of $\Omega\times(0,T]$.

Our motivation for the study of \eqref{eq:SD-PME} is the biofilm growth model introduced and numerically studied in \cite{Eberl2000} and rigorously analysed in \cite{Efendiev2009}.
Biofilms are communities of microorganisms in a moist environment in which cells stick to each other and often to a surface as well.
These cells are embedded in a slimy matrix of extracellular polymeric substances produced by the cells within the biofilm.
On the mesoscale, fully grown biofilms form complex heterogeneous shaped structures that the model in \cite{Eberl2000} is capable of predicting.
The model consists of two reaction-diffusion equations that are coupled via the reaction terms.
The first equation describes the biomass density $M$.
It is a quasilinear equation with a diffusion coefficient that vanishes whenever the biomass density is zero and blows up as the biomass density approaches its maximal value.
The second equation is a classical semilinear equation describing the growth limiting nutrient concentration $C$.
The variables are normalized; $C$ is scaled with respect to the bulk concentration and $M$ with respect to the maximum biomass density. 
The biofilm growth model is given by the system	
\begin{equation}\label{eq:Biofilm}
	\left\{\begin{aligned}
		\partial_t M&=d_2\nabla\cdot\left(D(M)\nabla M\right)-K_2M+K_3\frac{CM}{K_4+C}&\\
		\partial_t C&=d_1\Delta C-K_1\frac{CM}{K_4+C}& 
	\end{aligned}\right.
	\quad \text{in}\ \Omega\times(0,T],
\end{equation}
where the biomass-dependent diffusion coefficient $D$ is given by 
\begin{align*}
	D(M)=\frac{M^b}{(1-M)^a}, \qquad a\geq 1, \ b>0.
\end{align*}
The constants $d_1$, $d_2$ and $K_4$ are strictly positive and $K_1$, $K_2$ and $K_3$ are non-negative.

The biomass diffusion coefficient $D$ has a degeneracy in $0$ known from the porous medium equation.
Moreover, $D$ becomes singular as $M$ approaches $1$ so that spatial spreading becomes very large whenever $M$ is close to $1$. This ensures, heuristically, that the biomass density remains bounded by its maximum value.
As a consequence, we do not need any boundedness assumption for the reaction terms.
Finally, observe that the equation for the biofilm density is included in the class of equations we study.
Indeed, by setting
\begin{equation}\label{eq:nonlin.Biofilm}
	\nonlin(\solA)=\int_0^\solA\frac{z^b}{(1-z)^a}\md z
\end{equation}
we recognize that the first equation of \eqref{eq:Biofilm} is a particular case of \eqref{eq:SD-PME}.

In \eqref{eq:Biofilm} the actual biofilm is the subregion of $\Omega$ where $M$ is positive, that is,
\[
\Omega_{M}(t)=\{x\in\Omega\mid M(x,t)>0\}.
\]
This region and its boundary are well-defined provided that the function $M$ is continuous.
Therefore, the study of continuity of solutions is of fundamental importance for the viability of the biofilm model.
Moreover, H\"older regularity of solutions ensures the convergence of certain more efficient numerical schemes as well.

The main result of this paper is that any weak solution $\solA$ of \eqref{eq:SD-PME} is H\"older continuous in the interior of $\Omega\times(0,T]$, provided that $\solA$ is bounded away from $1$.
This builds upon our earlier work \cite{HissMull-Sonn22}, where we showed the well-posedness of \eqref{eq:SD-PME} and \eqref{eq:Biofilm} on bounded Lipschitz domains for initial and boundary value problems with mixed boundary conditions.
Indeed, we infer from our main result that the solutions of \eqref{eq:SD-PME} obtained in \cite{HissMull-Sonn22} are H\"older continuous in the interior of the domain.
The same holds for the first equation of \eqref{eq:Biofilm} implying interior H\"older continuity for $M$. 
Moreover, the second equation of \eqref{eq:Biofilm} is non-degenerate so classical results on interior H\"older continuity can be applied to obtain interior H\"older continuity for $C$, see for example \cite{lady1968}.

\Cref{eq:SD-PME} falls within a larger class of equations for which a weaker regularity result is available.
By setting $\beta=\nonlin^{-1}$ we see that \eqref{eq:SD-PME} is an example of a Stefan problem, which is an equation of the form
\begin{equation} \label{eq:Stefan}
	\partial_t\beta(\solB)=\Delta\solB + f(\emptyarg,\beta(\solB))
\end{equation}
for some non-decreasing $\beta:\IR\to\IR$.
Bounded classical solutions of \eqref{eq:Stefan} have a modulus of continuity in the interior of $\Omega_T$ depending only on $\norm{\solB}_{\infty}$ and the structure of the equation, see \cite{Sacks83}.
This modulus of continuity carries over to bounded weak solutions provided that they can be approximated by classical solutions.

The techniques we use to study \eqref{eq:SD-PME} can also be applied to the corresponding Stefan problem \eqref{eq:Stefan} to conclude interior H\"older continuity for $\solB=\nonlin(\solA)$.
Then, H\"older continuity of $\solA$ can be inferred provided that $\beta:=\nonlin^{-1}$ is H\"older continuous.
This was the original method applied to the porous medium equation, see \cite{DiBen1984}.
In fact, this is the natural approach, since a typical property of Stefan problems is that their solutions $\solB$ enjoy better continuity properties than $\solA=\beta(\solB)$, see \cite{Sacks83}.
Therefore, if H\"older continuity for $\nonlin(\solA)$ cannot be deduced, then it is unlikely that H\"older continuity for $\solA$ can be obtained.
This can be made rigorous if $\nonlin$ and $\beta:=\nonlin^{-1}$ are both H\"older continuous. 
However, in our case $\nonlin$ is not H\"older continuous due to the singularity $\nonlin(1)=\infty$.
We remark that we chose to study the original equation \eqref{eq:SD-PME}, since it describes the quantity of interest in applications.

We show interior H\"older continuity for solutions $\solA$ of \eqref{eq:SD-PME} that are bounded away from $1$ and this requirement is to be expected in view of the interpretation of \eqref{eq:SD-PME} as a Stefan problem of the form \eqref{eq:Stefan}.
Indeed, observe that this requirement for $\solA$ is equivalent to stating that the solution $\solB=\nonlin(\solA)$ is bounded.
The latter condition is needed to obtain a modulus of continuity for $\solB$, see \cite{Sacks83}, which is a weaker property than H\"older continuity.
Moreover, the boundedness of $\solB$ assumed in \cite{Sacks83} is a necessary condition; it can be deduced from assuming a modulus of continuity on $\solB$.
Returning to the original equation \eqref{eq:SD-PME}, we note that the condition on $\solA$ is not restrictive either; in \cite{HissMull-Sonn22} we have shown that solutions are bounded away from $1$ provided that the initial data and Dirichlet boundary conditions satisfy this assumption, which is a standard assumption in the typical problems considered in applications.

Interior H\"older continuity of solutions of non-degenerate parabolic equations is well-understood.
Moser adapted the techniques of DeGiorgi for uniformly elliptic linear equations for uniformly parabolic linear equations by studying the oscillation of solutions in a family of shrinking space-time cylinders reflecting the scaling invariance of the equation, see \cite{Mos60,Mos64}.
These methods where subsequently generalized by Lady{\ifmmode\check{z}\else\v{z}\fi}enskaja, Solonnikov and Ural'ceva for quasi-linear non-degenerate parabolic equations, see \cite{lady1968}.

Our study of regularity of solutions of \eqref{eq:SD-PME} uses so-called \emph{intrinsic scaling} techniques.
It was originally developed by DiBenedetto and Friedman in order to prove interior H\"older continuity for the $p$-Laplace equation and the porous medium equation, see \cite{DiBen1984}.
These techniques are widely used, see for example \cite{DiBenedetto1993,DiBenedetto2012,Urb08} and the references therein.
The key idea of intrinsic scaling is to change the scaling of the aforementioned space-time cylinders in a manner that depends on the solution itself.
Although some technical details are needed, the method allows us to carry over the techniques for non-degenerate parabolic equations to degenerate equations.

In the study of the regularity of solutions of porous medium type equations it is often assumed that solutions can be constructed as weak limits of sufficiently smooth solutions to a regularized problem, see for example \cite{DiBenedetto1983,DiBenedetto-Vespri1995}, or it is assumed that the solution is a classical solution, see \cite{Sacks83}.
This is done to justify some of the computations and it is always stressed that the obtained modulus of continuity solely relies on the data and is independent of the approximation.
This assumption is not considered to be restrictive in view of typical existence proofs and is therefore often omitted, for instance, see \cite{DiBen1984,Urb08}.

On the other hand, an important class of solutions is constructed by a finite-time discretization and Galerkin approximation scheme, see \cite{Alt-Luck1983}, which is needed when dealing with mixed Neumann-Dirichlet boundary conditions on Lipschitz domains.
The well-posedness proof in \cite{HissMull-Sonn22} is based on these methods.
For this class of solutions it is not immediately clear that the approximation assumption holds, since it does not follow from the existence proof presented in \cite{Alt-Luck1983} in a natural manner.
It can be shown that there exists a sequence of approximate continuous solutions, which would be sufficient for our purposes.
Nevertheless, this approach relies on the availability of the well-posedness of an appropriate initial- / boundary value problem rather than invoking only the definition of local solutions.

In this paper, we opt for a self-contained approach and ask for slightly more time regularity in our solution concept than usual in the literature on regularity theory, see \cite{DiBenedetto1993,DiBenedetto2012,lady1968,Liao19,Urb08}.
We assume that locally the time derivative of a solution is given as a bounded linear functional on the space of test functions.
It allows us to prove a chain rule for the term involving the time derivative used to obtain the necessary estimates.
In particular, we do not assume that solutions are the limit of sufficiently regular solutions of appropriate approximate problems.
Our assumption is not restrictive and is in line with the solution concept used in \cite{Alt-Luck1983} and \cite{HissMull-Sonn22}.

Interestingly, the assumption that solutions are weak limits of more regular solutions is not needed when studying the $p$-Laplace equation $u_t=\mathrm{div}(\abs{\nabla\solA}^{p-2}\nabla\solA)$.
Indeed, the natural functional space used in the solution concept is $L^p(0,T;W^{1,p}(\Omega))$.
For this space smoothing techniques such as Steklov averaging apply in a straightforward manner, because the gradient of the solution exists.
This does not hold for porous medium type equations, since only the existence of the gradient of $\nonlin(\solA)$ is assumed.

On the reaction term in \eqref{eq:SD-PME} we impose that it is locally bounded with respect to $\solA \in (0,\infty)$ and allow for a singularity in zero. 
In particular, we ask that $\norm{\source(\emptyarg,z)}_\infty \leq \LipBoundSource z^{-m_0}$, $\LipBoundSource\geq 0$, for all $z\in[0,1)$ for a certain $m_0>0$ depending on the structure of the degeneracy of $\nonlin$ in zero.
Motivated by the application to the biofilm growth model and to simplify the analysis, we had first imposed that the reaction term is bounded on the interval $[0,1)$.
We were able to show interior H\"older continuity under this assumption.
Later, we observed that our proof allowed for even more general reaction terms that might be singular in $0$, so we include those in our result.
It differs from standard assumptions in the literature on regularity theory, see \cite{DiBen1984,Sacks83,Urb08,DiBenedetto2012}.
We can show that the more general solution concept used in \cite{DiBenedetto2012} has a time derivative that can be interpreted as a bounded linear functional and therefore, is included in our solution concept, provided that $\source(\emptyarg,\solA)\in L^2_\loc(\Omega\times (0,T))$.
For details, see Proposition 3.19 in \cite{HissMull-Sonn22}.

Our proof follows closely the original proof of interior H\"older continuity for the model case $\solA_t=\Delta\solA^m$ in \cite{DiBen1984}.
Still, we do need to modify the proof to accommodate the reaction term.
For instance, we replace the typical `oscillation is large' estimate $\omega^{m-1}\geq R^\varepsilon$ by $\omega^m\geq R$, see \eqref{eq:oscillation.Really.Large} below.
Consequently, we also change some of the parameters in the subsequent iterative scheme.
Again, note that we do not rewrite the equation as a Stefan problem and prove H\"older continuity of $\solB=\nonlin(\solA)$, as was done in \cite{DiBen1984}.
Instead, we work with \Cref{eq:SD-PME} itself.
This leads to some difficulty in obtaining local energy estimates near the degeneracy in $\solA=0$.
We solve this issue in the same manner as was done in \cite{Urb08}, Chapter 6.

There is a different approach to obtain H\"older continuity for degenerate equations using \emph{expansion of positivity}, see  \cite{Liao19}.
The main feature of this method is that neither the logarithmic integral estimate nor the analysis of two alternatives in the key steps of the original proof are needed.
The expansion of positivity is a fundamental ingredient in proving a Harnack inequality, see \cite{DiBenedetto2012}.
We refrain from using these techniques so that our proof stays self-contained.
Furthermore, it is also not clear whether this approach allows for the reaction term we are considering.

More recently, intrinsic scaling has been applied successfully to doubly nonlinear parabolic equations whose prototype is $\partial_t (\abs{\solA}^{p-2}\solA) = \mathrm{div}(\abs{\nabla \solA}^{p-2} \solA)$, $p>1$, see \cite{BOG21,BOG21Part2,Liao22}.
This equation is a combination of the porous medium equation and the $p$-Laplace equation.
The expansion of positivity plays a fundamental role in the proofs.

The outline of this paper is as follows.
In \Cref{sect:MainHypothesis.and.THM} the main hypotheses, the solution concept and the main result on interior H\"older continuity are stated.
\Cref{sect:Geometry.Holder} introduces some further notation and a geometric setting in which the corner stone of the proof of the main result is given, the so-called \emph{De Giorgi-type Lemma}.
Based on this lemma, we prove H\"older continuity of solutions via an iterative scheme of intrinsically scaled shrinking space-time cylinders.
The next two sections are dedicated to proving the De Giorgi-type Lemma, which is where the main technicalities lie.
In particular, \Cref{sect:Int.est.and.Aux} covers interior integral estimates and some auxiliary technical statements and \Cref{sect:Proof.of.DeGiorgiLemma} contains the actual proof.

\section{Assumptions and main result}\label{sect:MainHypothesis.and.THM}
Let $\Omega\subseteq \IR^N$ be open, $0<T\leq\infty$ and write $\Omega_T:=\Omega\times(0,T]$.
The function $\nonlin:[0,1)\to[0,\infty)$ satisfies the structural assumptions:
\newcounter{counterHypotheses}
\begin{enumerate}[label=\upshape(H\arabic*),ref=\upshape H\arabic*]
	\item $\nonlin$ is continuous and strictly increasing;  \label{itm:nonlin.basic.assumption}
	\item $\nonlin$ is surjective; \label{itm:nonlin.blow-up.assumption}
	\item $\nonlin\in C^1([0,1))$, $\nonlin'>0$ in $(0,1)$ and a \emph{porous-medium type degeneracy} holds, that is,  $\nonlin'$ satisfies
	\[
	c_1z^{m-1}\leq \nonlin'(z)\leq c_2z^{m-1}
	\]
	for all $z\in [0,\varepsilon]$ for certain constants $c_1,c_2>0$, $\varepsilon\in(0,1)$ and $m>1$. \label{itm:nonlin.PME-like.degeneracy}
	\setcounter{counterHypotheses}{\value{enumi}}
\end{enumerate}
Heuristically, \eqref{itm:nonlin.blow-up.assumption} encodes the singularity $\nonlin(1)=\infty$ and \eqref{itm:nonlin.PME-like.degeneracy} the degeneracy $\nonlin'(0)=0$.
The conditions \eqref{itm:nonlin.basic.assumption} and \eqref{itm:nonlin.blow-up.assumption} are used to prove well-posedness in \cite{HissMull-Sonn22}.
The newly added assumption \eqref{itm:nonlin.PME-like.degeneracy} provides a connection to the porous medium equation and it is instrumental to prove H\"older continuity.
Of course, \eqref{itm:nonlin.basic.assumption} is implied by \eqref{itm:nonlin.PME-like.degeneracy}, but we mention both assumptions, because they play different roles in the study of \eqref{eq:SD-PME}.

It is important to point out that $\nonlin'\geq \lambda > 0$ on $[\varepsilon,1)$ for some $\lambda>0$.
Indeed, \eqref{itm:nonlin.blow-up.assumption} implies that $\lim_{z\to1}\nonlin(z)=\infty$, and therefore $\lim_{z\to 1} \nonlin'(z) = \infty$ as well so that $\nonlin'$ attains a minimum in $[\varepsilon,1)$.
This minimum cannot be $0$ due to the assumption $\nonlin'>0$ in $(0,1)$.

On the reaction term $\source:\Omega_T\times [0,1)\to\IR$ we impose:
\newcounter{counterHypothesesReactionTerms}
\begin{enumerate}[label=\upshape(R\arabic*),ref=\upshape R\arabic*]
	\item\label{itm:source.Lipschitz.assumption} $\source$ is measurable and there exists a $\LipBoundSource\geq 0$ such that 
	\[
	\norm{\source(\emptyarg,z)}_{L^\infty(\Omega_T)}\leq \LipBoundSource z^{-{m_0}} \quad \text{for all}\ z\in[0,1),
	\]
	where $m_0\in[0,m)$ with $m$ as in \eqref{itm:nonlin.PME-like.degeneracy}.
	\setcounter{counterHypothesesReactionTerms}{\value{enumi}}
\end{enumerate}
In is important to remark that \eqref{itm:source.Lipschitz.assumption} is not restrictive.
Indeed, solutions take values in the interval $[0,1)$, hence reaction terms such as
\[
\source(x,t,\solA)=g(x,t) \solA \quad \text{and} \quad \source(\solA) = \solA^p(1-\solA)^q + c
\]
are included for any bounded function $g$ and constants $p>-m$, $q\geq 0$ and $c\in\IR$.
The first example appears in the equation for $M$ of the biofilm growth model \eqref{eq:Biofilm} (provided that the function $C$ is known).
The latter example is used in the Porous-Fisher equation $\solA_t=\Delta \solA^m +\solA(1-\solA)$, studied e.g.\ in \cite{mccue2019hole}.

We employ the following notation. 
Given a measurable set $K\subseteq\Omega$, let $\pinprod{\emptyarg}{\emptyarg}$ denote the pairing of $H^{-1}(K)$ with $H^1_0(K)$ and let $\inprod{\emptyarg}{\emptyarg}$ denote the $L^2(K)$-inner product.
We use the following solution concept.
\begin{definition}\label{def:local.solution}
	A measurable function $\solA:\Omega_T\to[0,1)$ is called a local solution of \eqref{eq:SD-PME} if for any compact subset $K \subset \Omega$ we have that
	\begin{enumerate}[label=(\roman*),font=\itshape]
		\item $\solA\in W^{1,2}_\loc(0,T;H^{-1}({K}))$, $\nonlin(\solA)\in L^2_\loc(0,T;H^1({K}))$ and $\source(\emptyarg,\solA)\in L_{\loc}^2(0,T;L^2({K}))$, and
		\item the identity
		\begin{equation}\label{eq:SD-PME.local.solution.including.time-derivative.id}
			\displayindent0pt
			\displaywidth\textwidth
			\begin{aligned}
				\pinprod{\solA_t}{\testA}+\inprod{\nabla\nonlin(\solA)}{\nabla\testA}=\inprod{\source(\emptyarg,\solA)}{\eta}
			\end{aligned}
		\end{equation}
		holds a.e.\ in $(0,T)$ for all $\eta\in H^1_0({K})$.
	\end{enumerate}
\end{definition}
Define the \emph{parabolic boundary} of $\Omega_T$ by $\Gamma=\overline{\Omega}\times\{0\}\cup \partial\Omega\times(0,T)$ and the \emph{parabolic distance} of compact set $K\subset \Omega_T$ to $\Gamma$ by
\[
\mathrm{dist}(K;\Gamma)=\inf\left\{\abs{x-y}+\abs{t-s}^{\frac{1}{2}}\, : \, (x,t)\in K,\ (y,s)\in \Gamma\right\}.
\]
The following statement is the main result of this paper.
\begin{theorem}[Interior H\"older continuity of solutions]\label{thm:holder}
	Let \eqref{itm:nonlin.basic.assumption}, \eqref{itm:nonlin.blow-up.assumption}, \eqref{itm:nonlin.PME-like.degeneracy} and \eqref{itm:source.Lipschitz.assumption} be satisfied and let $\solA$ be a local solution of \eqref{eq:SD-PME} that is bounded away from $1$, that is, there exists a $\mu \in(0,1)$ such that $\solA\leq 1-\mu$.
	Then, there exist constants $C\geq 0$ and $\alpha\in(0,1)$ depending only on $N$, $c_1$, $c_2$, $\varepsilon$, $m$, $\LipBoundSource$, $m_0$, ${M}:=\max_{0\leq z\leq 1-\mu}\nonlin'(z)$ and $\lambda:=\min_{\varepsilon\leq z<1}\nonlin'(z)$ such that
	\begin{equation}\label{eq:holder}
		\abs{\solA(x_0,t_0)-\solA(x_1,t_1)}\leq C\left(\frac{\abs{x_0-x_1}+\abs{t_0-t_1}^{\frac{1}{2}}}{d(K;\Gamma)}\right)^\alpha
	\end{equation}
	for all $(x_0,t_0),(x_1,t_1)\in K$ for any compact $K\subset\Omega_T$, where $d(K;\Gamma):=\min \left\{ \mathrm{dist}(K;\Gamma),1 \right\}$.
\end{theorem}

\begin{remark}\label{rem:reaction.term.cond}
	Let us provide some remarks on the assumption \eqref{itm:source.Lipschitz.assumption} for the reaction term.
	\begin{itemize}
		\item The condition \eqref{itm:source.Lipschitz.assumption} has not been considered in the literature to the author's knowledge.
		The condition requires some changes in the classical proof for the prototype porous medium equation $\solA_t=\Delta\solA^m$ in \cite{DiBen1984}.
		It is more general than what we would need in view of the application to the biofilm model \eqref{eq:Biofilm} we have in mind.
		\item Conditions on the reaction terms for the porous medium equation that have been covered include the assumption
		\[
		\abs{\source(\emptyarg,z)}\leq \varphi_1
		\]
		for some non-negative function $\varphi_1\in L^{q}(0,T;L^p(\Omega))$, see page 48 in \cite{Urb08}, and the assumption
		\[
		\abs{\source(\emptyarg,z)}\leq \abs{\solA}^m\varphi_2,
		\]
		for some non-negative function $\varphi_2$ such that $\varphi_2^2\in L^{q}(0,T;L^p(\Omega))$, see page 261 in \cite{DiBenedetto2012}.
		In both cases $p$ and $q$ satisfy $\frac{1}{q}+\frac{N}{2p}\in (0,1)$.
		\item The condition \eqref{itm:source.Lipschitz.assumption} is not covered by our previous work on the well-posedness of \eqref{eq:SD-PME}, see \cite{HissMull-Sonn22}.
		However, the condition that $\source$ is bounded on $\Omega_T\times[0,1)$ is included.
		In particular, the biofilm growth model \eqref{eq:Biofilm} is covered by in both the well-posedness result and our current result on H\"older continuity.
	\end{itemize}
\end{remark}

\begin{remark}\label{rem:extending.PME-like.degeneracy}
	Let us discuss some additional observations regarding the hypothesis \eqref{itm:nonlin.PME-like.degeneracy}.
	\begin{itemize}
		\item Without loss of generality, we may assume that $\varepsilon \leq 1-\mu$ in \eqref{itm:nonlin.PME-like.degeneracy} so that the solution $\solA$ in \Cref{thm:holder} satisfies
		\begin{equation}\label{eq:extending.PME-like.degeneracy}
			c_1 \solA^{m-1}\leq \nonlin'(\solA)\leq c_2 \solA^{m-1}.
		\end{equation}
		Indeed, suppose $\varepsilon > 1-\mu$ and observe that the function $\nonlin'$ satisfies $\nonlin'\leq {M}$ on $[\varepsilon , 1-\mu]$ and $\nonlin'\geq \lambda$ on $[\varepsilon,1)$.
		There certainly exist constants $d_1$, $d_2>0$ such that $d_1z^{m-1}\leq\nonlin'(z)\leq d_2 z^{m-1}$ for all $z\in [\varepsilon,1-\mu]$.
		For instance, set $d_1={M}$ and $d_2=\lambda / \varepsilon$.
		We pick new constants $\tilde{c}_1=\min\{c_1,d_1\}$ and $\tilde{c}_2=\min\{c_2,d_2\}$, which depend on $c_1$, $c_2$, $\varepsilon$, ${M}$ and $\lambda$.
		Now \eqref{itm:nonlin.PME-like.degeneracy} holds with $c_1$, $c_2$ and $\varepsilon$ replaced by $\tilde{c}_1$, $\tilde{c}_2$ and $1-\mu$, respectively.
		\item We point out that the previous argument is the only place where $\epsilon$, ${M}$ and $\lambda$ play a role in the proof of \Cref{thm:holder}.
		We will always assume that $\varepsilon\leq 1-\mu$ in \eqref{itm:nonlin.PME-like.degeneracy} holds and we do not mention the dependency of the constants on $\varepsilon$, ${M}$ and $\lambda$ from now on.
		The symbols $\varepsilon$ and $\lambda$ are used again below in different contexts, which is justified by this remark.
		\item Observe that \eqref{itm:nonlin.PME-like.degeneracy} implies that
		\begin{equation}\label{eq:nonlin.PME-like.degeneracy.integrated}
			c_1z^{m}\leq m\nonlin(z)\leq c_2z^{m}.
		\end{equation}
		This can be seen by integrating the estimate over $z$, multiplying by $m$ and observing that \eqref{itm:nonlin.basic.assumption} and \eqref{itm:nonlin.blow-up.assumption} imply $\nonlin(0)=0$.
	\end{itemize}
\end{remark}

\begin{remark}\label{rem:defined.all.time}
	It is important to point out that $\solA(t)\in L^1(K)$ is well-defined for all $t\in(0,T)$.
	Indeed, for any $t_1,t_2\in(0,T)$, $t_1<t_2$, we have that $\solA\in W^{1,2}(t_1,t_2;H^{-1}(K))\subset C([t_1,t_2];H^{-1}(K))$.
	Therefore, $\solA(t)\in L^1(K)\cap L^\infty(K) \subset L^2(K) \subset H^{-1}(K)$ is uniquely determined.
	Moreover, $\solA\in L^\infty(t_1,t_2;L^1(K))$ with $\norm{\solA(t)}_{L^1(K)}\leq \abs{K}$ for all $t\in(0,T)$ due to $0\leq\solA<1$.	
\end{remark}

\begin{remark}
	Let us also comment on the use of the parabolic distance.
	\begin{itemize}
		\item In the literature, see \cite{DiBenedetto2012} or \cite{Urb08}, the intrinsic parabolic $m$-distance is introduced to quantify the dependency of the constants of \Cref{thm:holder} on the compact subset $K\subset\Omega_T$ and $\norm{\solA}_{L^\infty(\Omega_T)}$.
		In our case we know that $0\leq\solA\leq 1$, so the standard parabolic distance suffices.
		\item In the denominator of \eqref{eq:holder} we use $d(K;\Gamma)$ instead of $\mathrm{dist}(K;\Gamma)$.
		This is due to the lower order term in the equation.
		To justify the estimates in the computations, we need to bound the size of the appropriate space-time cylinders.
		Without the reaction term the full parabolic distance $\mathrm{dist}(K;\Gamma)$ can be used and a Liouville-type result can be deduced.
		However, this is not possible in our case.
	\end{itemize}	
\end{remark}

\begin{remark}
	We mention a few additional points concerning the condition that the solution has to be bounded away from $1$.
	\begin{itemize}
		\item We can replace it by the requirement that the solution is \emph{locally} bounded away from $1$, that is, for every compact subset $K\subset\Omega_T$ there exists a $\mu$ such that $\solA\leq 1-\mu$ in $K$.
		In this case, $C$ and $\alpha$ in \Cref{thm:holder} do depend on $K$.
		\item The assumption is not restrictive and builds upon our earlier work \cite{HissMull-Sonn22}.
		There we showed that $\solA \leq 1-\mu$ provided that the initial data $\solA_0$ satisfies $\solA_0\leq 1-\theta$, where $\mu$ depends on $\theta$.
		This argument relies on the blow-up behaviour of $\nonlin$ encoded in \eqref{itm:nonlin.blow-up.assumption} and on the domain being bounded and having a Lipschitz boundary.
		Now, the interior H\"older continuity of $\solA$ only depends on the initial data $\solA_0$ through $\mu$.
		\item The assumption implies that the singularity of $\nonlin$ in \eqref{itm:nonlin.blow-up.assumption} is not attained.
		Indeed, we assume that $\solA\leq 1-\mu$ so the behaviour of $\nonlin$ in $(1-\mu,1)$ does not matter, and therefore we could remove \eqref{itm:nonlin.blow-up.assumption} provided that we impose that $\nonlin'\geq \lambda>0$ on $[\varepsilon,1-\mu]$ for some $\lambda>0$ and $\nonlin(0)=0$.
		Consequently, our result also holds for degenerate equations without a singularity such as the Porous-Fisher equation, see \cite{mccue2019hole}.
		However, we aim to prove H\"older continuity for the solutions obtained in our earlier work on well-posedness, see \cite{HissMull-Sonn22}, where \eqref{itm:nonlin.blow-up.assumption} is a key ingredient to obtain the bound $\solA \leq 1-\mu$.
		Moreover, \eqref{itm:nonlin.blow-up.assumption} plays a fundamental role to obtain the uniform bound $\solA<1$ and it is necessary to be able to formulate $\eqref{eq:SD-PME}$ as a Stefan problem of the form \eqref{eq:Stefan}, since otherwise $\nonlin$ might not be invertible. 
		Therefore, we keep the hypothesis \eqref{itm:nonlin.blow-up.assumption}.
	\end{itemize}
\end{remark}
We use the following notation throughout.
Given $k\in\IR$, we write
\[
\left[\solA < k \right] = \left\{(x,t)\in \Omega_T \, :\, \solA(x,t)<k \right\}\quad \text{and} \quad \left[\solA(t) < k\right] = \left\{ x \in \Omega \, :\, \solA(x,t)<k \right\}
\]
for fixed $t\in(0,T)$ and we extend the notation for $\leq$, $>$ and $\geq$ in an obvious manner.
Both sets are defined up to a subset of measure zero, the latter for all $t\in(0,T)$ by \Cref{rem:defined.all.time}, hence their measures are well-defined and their characteristic functions exist almost everywhere.	
	
\section{Geometry for the equation and proof of the main result}\label{sect:Geometry.Holder}
In this section we prove \Cref{thm:holder} based on an intermediary result called the De Giorgi-type Lemma, which we prove in \Cref{sect:Proof.of.DeGiorgiLemma}.
The outline of this section is as follows.
First, we discuss the method of \emph{intrinsic scaling} heuristically.
Then, we rigorously define the intrinsic geometrical scaling tailored for our equation, where we introduce the necessary notation and assumptions concerning intrinsically scaled space-time cylinders.
After this we state the De Giorgi-type Lemma.
Next, we provide an iterative scheme where we consider a shrinking sequence of the cylinders and we show that the oscillation of $\solA$ in each cylinder decreases proportionally to its size using the De Giorgi-lemma iteratively.
In the next step of the proof of \Cref{thm:holder} we consider the case where the De Giorgi-lemma cannot be applied.
In this situation we can use classical estimates for non-degenerate equations.
Finally, we combine both cases in one final Corollary, which we then use to prove \Cref{thm:holder}.

Let us begin with a short informal discussion on intrinsic scaling.
The proof of interior H\"older continuity for non-degenerate parabolic equations relies on estimates of the oscillation of the solution $\solA$ in a family of shrinking space-time cylinders $\{Q_n\}_{n=0}^\infty$, see \cite{lady1968}.
These cylinders are roughly of the form $Q_n=B_{R_n}(x_0) \times (-R_n^2+t_0,t_0)$ and their space-time scaling reflects the structure of the equation.
For instance, the heat equation $\solA_t=\Delta\solA$ is scale invariant under coordinate transformations of the type $(x,t)\mapsto (\lambda x,\lambda^2 t)$, $\lambda\in\IR$, that is, if $\solA$ is a solution of the equation then $\solA_\lambda$ given by $\solA_\lambda(x,t)=\solA(\lambda x,\lambda^2,t)$ is a solution as well.
The method of shrinking cylinders with this parabolic scaling fails to hold for degenerate equations such as the porous medium equation $\solA_t=\Delta \solA^m$.
However, we can rewrite this equation as
\[
\frac{1}{m} \frac{1}{\solA^{m-1}}(\solA^m)_t=\Delta \solA^m,
\]
which is a non-degenerate equation except for the factor $\solA^{1-m}$.
Therefore, we should consider space-time cylinders of the form $Q_n=B_{R_n}(x_0) \times (-\solA^{1-m}R^2_n+t_0,t_0)$ instead.
These cylinders are defined in terms of the solution itself so we say that they are intrinsically scaled.
The techniques for non-degenerate equations can be applied along this set of shrinking cylinders, since we now consider the equation in its own geometry.
This method is only relevant if $\solA$ is small, because then the equation becomes degenerate.
Consequently, the intrinsically scaled cylinders are stretched along the temporal axis to \emph{accommodate the degeneracy}.
\Cref{fig:intrinsicscalingsolplotwithcylinders} highlights the difference in the types of scaling.
Here, $\solA$ is a solution of \eqref{eq:SD-PME} and we consider points $(x_0,t_0)$ and $(\tilde{x}_0,\tilde{t}_0)$.
We pick $(x_0,t_0)$ near the degeneracy and $(\tilde{x}_0,\tilde{t}_0)$ away from the degeneracy.
In the first case we use intrinsically scaled cylinders and in the latter case we use classically scaled cylinders.

\begin{figure}[ht]
	\centering
	\includegraphics[width=0.8\linewidth]{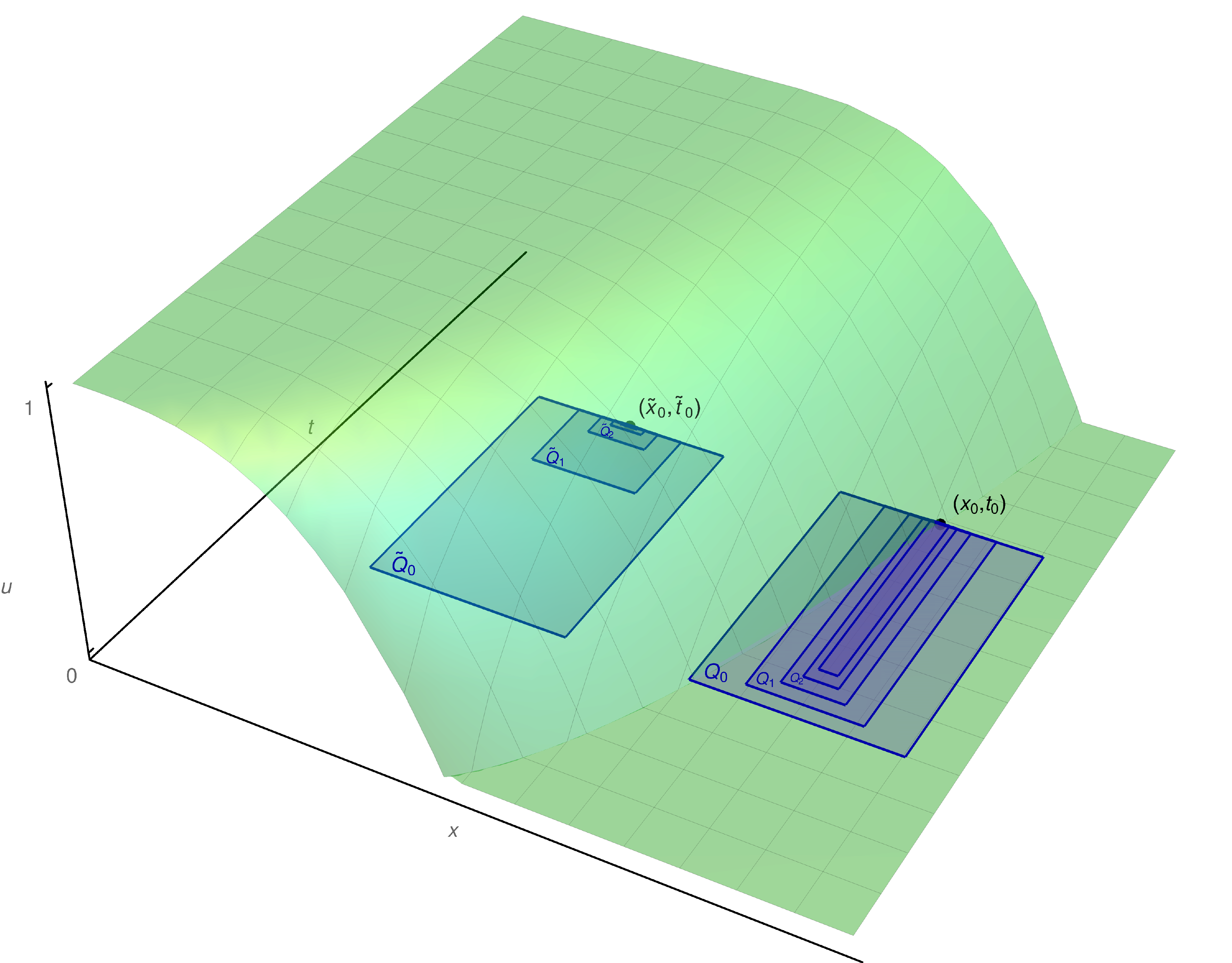}
	\caption{Intrinsically scaled and classically scaled cylinders.}
	\label{fig:intrinsicscalingsolplotwithcylinders}
\end{figure}

Let us now provide the intrinsic scaling in a rigorous manner.
First, assume that $\solA$ is a local solution of \eqref{eq:SD-PME} that is bounded away from $1$.
Let $(x_0,t_0)\in \Omega_T$ and assume that $(x_0,t_0)=(0,0)$ by translating the axes, where we note that only the reaction term $\source$ is not invariant under such a translation, but we only rely on global estimates of this term, hence the assumption is justified.
Given $R_1,R_2>0$ we define the space-time cylinder
\begin{equation*}
	Q(R_1,R_2)=B_{R_1}(0)\times(-R_2,0),
\end{equation*}
and given $\omega,R>0$ we define the \emph{intrinsically scaled cylinder}
\[
Q_\omega(R):=Q(R,\omega^{1-m}R^2).
\]
For now we assume that
\begin{equation}\label{eq:cond.intrinsic.scaled.cylinder.in.domain}
	Q_\omega(R)\subseteq Q_\omega(2R)\subseteq\Omega_T.
\end{equation}
We define $\mu_-,\mu_+\in[0,1-\mu]$ and the essential oscillation of $\solA$ by
\[
\mu_-:=\essinf_{Q_\omega(R)}\solA,\quad \mu_+:=\esssup_{Q_\omega(R)}\solA,\quad\essosc_{Q_\omega(R)}\solA:=\mu_+-\mu_-.
\]
We assume that
\begin{equation}\label{eq:cond.oscillation.recursion.hypothesis}
	\essosc_{Q_\omega(R)}\solA\leq \omega,
\end{equation}
so that $\omega$ can be viewed as (an upper bound of) the essential oscillation of $\solA$ in $Q_\omega(R)$.
Further, we assume that
\begin{equation}\label{eq:cond.inf.small}
	\mu_-\leq \frac{\omega}{4}
\end{equation}
holds as well.
Note that \eqref{eq:cond.inf.small} implies that $\mu_+\leq \omega+\mu_-\leq \frac{5}{4}\omega$.

\begin{remark}
	Heuristically, \eqref{eq:cond.inf.small} means that $\solA$ takes values close to $0$ and this specific estimate gives us a means to quantify this in terms of the oscillation.
	We use it primarily to estimate $\mu_+$ in terms of $\omega$.
	If \eqref{eq:cond.inf.small} does not hold, then $\solA$ is bounded away from $0$ in $Q_\omega(R)$.
	In this case we are effectively dealing with a non-degenerate equation and we can invoke standard estimates, see the text below \Cref{prop:iterative.scheme}.
\end{remark}

We use the notation
\[
Q^{\nu_0}_\omega(R)=Q\left(R,\frac{\nu_0}{2}\omega^{1-m}R^2\right),
\]
for a given $\nu_0\in(0,1)$.

\begin{proposition}[De Giorgi-type Lemma]\label{lem:DeGiorgi-type}
	Assume that \eqref{itm:nonlin.basic.assumption}, \eqref{itm:nonlin.blow-up.assumption}, \eqref{itm:nonlin.PME-like.degeneracy} and \eqref{itm:source.Lipschitz.assumption} are satisfied and let $\solA$ be a local solution of \eqref{eq:SD-PME} that is bounded away from $1$.
	Then, there exist constants $R_{\mathrm{max}},\nu_0\in(0,1)$ and $n_0\in\IN$ depending only on $N$, $c_1$, $c_2$, $m$, $\LipBoundSource$ and $m_0$ such that the following holds.
	If $R\in(0,R_{\mathrm{max}}]$ and $\omega>0$ are such that \eqref{eq:cond.intrinsic.scaled.cylinder.in.domain}, \eqref{eq:cond.oscillation.recursion.hypothesis}, \eqref{eq:cond.inf.small} and
	\begin{equation}\label{eq:oscillation.Really.Large}
		\omega\geq R^{\frac{1}{m}}
	\end{equation}
	are satisfied, then we have the dichotomy:
	\begin{enumerate}[label=(\roman*),font=\itshape]
		\item If
		\begin{equation}\label{eq:DeGiorgi-type.alternative.I}
			\abs{Q_\omega(R)\cap\left[\solA<\mu_-+\frac{\omega}{2}\right]}<\nu_0\abs{Q_\omega(R)},
		\end{equation}
		then $\solA>\mu_-+\frac{\omega}{4}$ a.e.\ in $Q_\omega(\frac{R}{2})$.
		\item If
		\begin{equation}\label{eq:DeGiorgi-type.alternative.II}
			\abs{Q_\omega(R)\cap\left[\solA\geq\mu_-+\frac{\omega}{2}\right]}\leq(1-\nu_0)\abs{Q_\omega(R)},
		\end{equation}
		then $\solA<\mu_-+\left(1-\frac{1}{2^{n_0}}\right)\omega$ a.e.\ in $Q^{\nu_0}_\omega(\frac{R}{2})$.
	\end{enumerate}
\end{proposition}
\Cref{sect:Proof.of.DeGiorgiLemma} is dedicated to proving \Cref{lem:DeGiorgi-type}.
The constants $R_{\mathrm{max}}$ and $\nu_0$ are defined explicitly by \eqref{eq:define.Rmax} and \eqref{eq:def.nu_0}, respectively, in terms of $N$, $c_1$, $c_2$, $m$, $\LipBoundSource$ and $m_0$.

\begin{remark}
	Comparing to previous results, see \cite{Urb08}, we add condition \eqref{eq:oscillation.Really.Large} in \Cref{lem:DeGiorgi-type}.
	We also introduce $R_{\mathrm{max}}$ to ensure that $R$ is small enough in a certain sense.
	Both are needed to derive suitable estimates for the reaction term in the proof of \Cref{lem:DeGiorgi-type}.
	This is necessary due to the assumption \eqref{itm:source.Lipschitz.assumption}, which differs from the standard assumptions.
\end{remark}

\begin{remark}
	\Cref{lem:DeGiorgi-type} provides a way to quantitatively improve the oscillation of the solution in a smaller cylinder, a role which is played by the Harnack inequality in the analysis of the heat equation.
	Informally speaking, \Cref{lem:DeGiorgi-type} states that if $\solA$ mostly takes values in the upper half / lower half of $[\mu_-,\mu_+]$, then $\solA$ is bounded away from $\mu_-$ / from $\mu_+$, respectively, in a smaller cylinder.
\end{remark}

Assume \Cref{lem:DeGiorgi-type} is proven.
We set 
\begin{equation}\label{eq:defining.improvement.of.oscillation.parameter}
	{\eta_0} = \max\left\{ \frac{3}{4},1-2^{n_0} \right\}.
\end{equation}
By noting that $Q^{\nu_0}_\omega(R)\subset Q_\omega(R)$ we conclude
\begin{align}
	\essosc_{Q^{\nu_0}_\omega(\frac{R}{2})}\solA&\leq\left\{\begin{array}{cl}
		\mu_+-\mu_--\frac{\omega}{4}			&\quad\text{if \eqref{eq:DeGiorgi-type.alternative.I} holds},\\
		\mu_-+(1-\frac{1}{2^{n_0}})\omega-\mu_-	&\quad\text{if \eqref{eq:DeGiorgi-type.alternative.II} holds}
	\end{array}\right. \nonumber \\
	&\leq {\eta_0} \omega	\label{eq:improvement.of.oscillation.parameter.estimate}
\end{align}
for any $R\in(0,R_{\mathrm{max}}]$ and $\omega>0$ satisfying \eqref{eq:cond.intrinsic.scaled.cylinder.in.domain}, \eqref{eq:cond.oscillation.recursion.hypothesis}, \eqref{eq:cond.inf.small} and \eqref{eq:oscillation.Really.Large}.
Estimate \eqref{eq:improvement.of.oscillation.parameter.estimate} is the key implication of \Cref{lem:DeGiorgi-type}; it shows that the oscillation of $\solA$ is improved in a smaller cylinder.

We proceed with the proof of \Cref{thm:holder} by considering a decreasing sequence of numbers $\{\omega_n\}_{n=0}^\infty$ and of shrinking intrinsically scaled space-time cylinders $\{Q_n\}_{n=0}^\infty$ along which we can iteratively apply \Cref{lem:DeGiorgi-type}.
It allows us to describe the oscillation of $\solA$ in $Q_n$ in terms of the radius of $Q_n$, which is an essential ingredient to prove H\"older continuity.
For the starting cylinder we assume that \eqref{eq:cond.intrinsic.scaled.cylinder.in.domain}, \eqref{eq:cond.oscillation.recursion.hypothesis} and \eqref{eq:oscillation.Really.Large} hold.
Then, we show that \eqref{eq:cond.intrinsic.scaled.cylinder.in.domain}, \eqref{eq:cond.oscillation.recursion.hypothesis} and \eqref{eq:oscillation.Really.Large} are satisfied in each subsequent step of the iteration.
Only \eqref{eq:cond.inf.small} might fail to hold at some step, in which case we have to stop the iterative procedure.
This leads to the split into two cases in the following proposition.

\begin{proposition}[The iterative scheme]\label{prop:iterative.scheme}
	Suppose $R_0\in(0,R_{\mathrm{max}}]$ and $\omega_0>0$ such that \eqref{eq:cond.intrinsic.scaled.cylinder.in.domain}, \eqref{eq:cond.oscillation.recursion.hypothesis} and \eqref{eq:oscillation.Really.Large} hold (substituting $R=R_0$ and $\omega=\omega_0$).
	Define sequences $\{R_n\}_{n=1}^\infty$ and $\{\omega_n\}_{n=1}^\infty$ recursively by
	\[
	R_{n+1}=aR_n,\quad\omega_{n+1}=\eta_0 \omega_n,
	\]
	where ${\eta_0}$ is given by \eqref{eq:defining.improvement.of.oscillation.parameter} and
	\begin{equation}\label{eq:defining.proportionality.R_n}
		a:=\frac{1}{2}\sqrt{\frac{\nu_0}{2}} \eta_0^{m},
	\end{equation}
	and define the cylinders $Q_n=Q_{\omega_n}(R_n)$.
	Then at least one of the following statements holds.
	\begin{enumerate}[label=(\roman*), font=\itshape]
		\item The estimate
		\begin{equation}\label{eq:oscillation.intrisinc.scales.with.cylinder}
			\essosc_{Q_{\omega_n}(R_n)}\solA\leq \omega_0 \left(\frac{R_n}{R_0}\right)^\alpha
		\end{equation}
		holds for all $n\in\IN$, where $\alpha\in(0,1)$ is given by
		\begin{equation}\label{eq:definition.holder.exp}
			\alpha=\frac{\log(\eta_0)}{\log(a)}.
		\end{equation}
		\item There exists an integer $n_*\in\IN$ such that \eqref{eq:oscillation.intrisinc.scales.with.cylinder} holds for all $n\leq n_*$ and the estimate
		\begin{equation}\label{eq:cond.inf.large}
			\essinf_{Q_{{n_*}}}\solA\geq \frac{1}{4}\omega_{n_*} 
		\end{equation}
		holds.
	\end{enumerate}
\end{proposition}
\begin{proof}
	By assumption, $\omega_0$ and $R_0$ satisfy \eqref{eq:cond.intrinsic.scaled.cylinder.in.domain}, \eqref{eq:cond.oscillation.recursion.hypothesis} and \eqref{eq:oscillation.Really.Large}.
	Suppose that $\omega_0$ and $R_0$ satisfy \eqref{eq:cond.inf.small} as well, because otherwise case \textit{(ii)} holds and there is nothing left to prove.
	
	First, we show that \eqref{eq:cond.intrinsic.scaled.cylinder.in.domain}, \eqref{eq:cond.oscillation.recursion.hypothesis} and \eqref{eq:oscillation.Really.Large} hold for $\omega_1$ and $R_1$.
	Certainly $R_1\leq R_2$ and we observe that
	\[
	\omega_1^{1-m}R_1^2\leq ({\eta_0}\omega_0)^{1-m}a^2R_0^2 = \frac{\nu_0}{2} \eta_0^{m+1} \omega_0^{1-m}\left(\frac{R_0}{2}\right)^2\leq \frac{\nu_0}{2}\omega_0^{1-m}\left(\frac{R_0}{2}\right)^2,
	\]
	hence $Q_{1}\subseteq Q^{\nu_0}_{\omega_0}\left(\frac{R_0}{2}\right)\subseteq Q_{0}$.
	Replacing $R_0$ and $R_1$ by $2R_0$ and $2R_1$ shows that $Q_{\omega_1}(2R_1)\subset Q_{\omega_0}(2R_0)$ and therefore \eqref{eq:cond.intrinsic.scaled.cylinder.in.domain} holds.
	Additionally, by \Cref{lem:DeGiorgi-type} we know that
	\[
	\essosc_{Q^{\nu_0}_{\omega_0}\left(\frac{R_0}{2}\right)}\solA \stackrel{\eqref{eq:improvement.of.oscillation.parameter.estimate}}{\leq} h\omega_0\leq\omega_1,
	\]
	so $\essosc_{Q_{1}}\solA\leq \omega_{1}$, i.e.\ \eqref{eq:cond.oscillation.recursion.hypothesis} holds.
	Finally, we compute that
	\[
	R_1=\frac{1}{2}\sqrt{\frac{\nu_0}{2}} \eta_0^{m}R_0\leq \eta_0^{m}R_0 \stackrel{\eqref{eq:oscillation.Really.Large}}{\leq} \eta_0^{m}\omega_0^{m}=\omega_1^{m},
	\]
	so \eqref{eq:oscillation.Really.Large} is valid for $\omega_1$ and $R_1$ as well.
	
	We can continue the arguments recursively as long as \eqref{eq:cond.inf.small} holds at each step, where we redefine
	\[
	\mu_-=\essinf_{Q_{n-1}}\solA
	\]
	at each step $n$.
	If this is the case, then we obtain the inclusions $Q_n\subset Q_{n-1}\subset\dots\subset Q_0\subset\Omega_T$ and the estimates $\essosc_{Q_{n}}\solA\leq \omega_n$ for each $n$ and it follows that
	\[
	\essosc_{Q_{n}}\solA\leq \omega_n=\eta_0^n\omega_0=\omega_0\left(\frac{R_n}{R_0}\right)^\alpha.
	\]
	Here, we used that $\left(R_n/R_0\right)^{\alpha}=a^{n\alpha}= \eta_0^n$.
	Indeed, $\alpha$ is given by \eqref{eq:definition.holder.exp}, hence $a^\alpha=\eta_0$.
	Moreover, $\alpha\in(0,1)$ due to $\frac{3}{4}\leq \eta_0 \leq 1$ and $\frac{1}{2}\sqrt{\frac{\nu_0}{2}} \eta_0^{m}\leq \frac{3}{4}$.
	This concludes case \textit{(i)} and the first statement of case \textit{(ii)}.
	
	Suppose that \eqref{eq:cond.inf.small} fails to hold at step $n_*+1$, then 
	\[
	\essinf_{Q_{n_*}}\solA\geq \frac{1}{4}\omega_{n_*},
	\]
	i.e.\ case \textit{(ii)} holds.
\end{proof}
\begin{remark}
	The factor $a$ given in \eqref{eq:defining.proportionality.R_n} is proportional to $\eta_0^m$, which differs from the factor used for the porous medium equation, where it is proportional to $\eta_0^{m-1}$.
	This difference is due to the assumption \eqref{itm:source.Lipschitz.assumption} on the reaction term.
	Indeed, we define $a$ by \eqref{eq:defining.proportionality.R_n} to make sure \eqref{eq:oscillation.Really.Large} is satisfied in each step of the iterative scheme and condition \eqref{eq:oscillation.Really.Large} is introduced due to this growth assumption.
	We could have picked $a$ proportional to $\eta_0^{m-1}$ and checked whether \eqref{eq:oscillation.Really.Large} holds at each step, but \eqref{eq:defining.proportionality.R_n} simplifies the arguments.
\end{remark}

To obtain the appropriate estimates for the oscillation of $\solA$ in $Q_{n_*}$ in case \textit{(ii)} of \Cref{prop:iterative.scheme} we use known results for uniformly parabolic quasi-linear equations to replace the role of \Cref{lem:DeGiorgi-type}.
This is possible, because \eqref{eq:cond.inf.large} implies that $u$ is bounded away from $0$ in $Q_{n_*}$, so \eqref{eq:SD-PME} is a non-degenerate quasilinear parabolic equation in $Q_{n_*}$.
First, we introduce a change of variables (stretching the space variable) to rewrite \eqref{eq:SD-PME} as a uniformly parabolic equation whose ellipticity condition does not depend on $\solA$ (through $\mu_{\pm}$ or $\omega_{n_*}$).
We set
\[
\bar{x}= \frac{x}{\sigma^{1/2} \rho},\quad \bar{t}= \frac{t}{\sigma},
\]
where $\rho^2:=\mu_-^{m-1}$, $\sigma:=\mu_-^{m_0}$ and $\mu_-:=\essinf_{Q_{n_*}}\solA\geq \frac{1}{4}\omega_{n_*}$ due to \eqref{eq:cond.inf.large}.
We also write $\bar{\solA}(\bar{x},\bar{t})=\solA(x,t)=\solA(\sigma^{1/2}\rho\bar{x},\sigma t)$, $\bar{f}(\bar{x},\bar{t},z)=f(x,t,z)=f(\sigma^{1/2}\rho\bar{x},\sigma t,z)$ and $\bar{R}= \sigma^{-1/2}\rho^{-1} {R}$, so that the cylinder $Q(R_1,R_2)$ transforms into $\bar{Q}(R_1,R_2)$ $:=$ $B_{\bar{R}_1}\times(-\sigma^{-1} R_2^2,0)$ $=$ $B_{\bar{R}_1}\times(-\rho^2 \bar{R}_2^2,0)$.
Moreover, \eqref{eq:SD-PME} transforms into $\sigma^{-1}\bar{\solA}_{\bar{t}}=\sigma^{-1}\rho^{-2}\Delta_{\bar{x}}\nonlin(\bar{\solA})+\bar{\source}(\emptyarg,\bar{\solA})$, which we multiply by $\sigma$ to obtain the equation
\begin{equation}\label{eq:SD-PME.stretched.space.variable}
	\bar{\solA}_{\bar{t}} = \frac{1}{\rho^2} \Delta_{\bar{x}}\nonlin(\bar{\solA})+\sigma\bar{\source}(\emptyarg,\bar{\solA}).
\end{equation}
The uniformly parabolic equation \eqref{eq:SD-PME.stretched.space.variable} has an ellipticity condition that only depends on $c_1$, $c_2$ and $m$, because
\[
\frac{\nonlin'(\bar{\solA})}{\rho^2}\geq c_1\frac{\mu_-^{m-1}}{\rho^2}= c_1,
\]
and
\begin{align*}
	\frac{\nonlin'(\bar{\solA})}{\rho^2}&\leq c_2\frac{\mu_+^{m-1}}{\rho^2}\leq c_2 5^{m-1}\frac{\mu_-^{m-1}}{\rho^2}=5^{m-1}c_2,
\end{align*}
where we used that $\mu_+:=\esssup_{Q_{n_*}}\solA\leq \omega_{n_*}+\mu_-\leq 5\mu_-$ due to \eqref{eq:cond.inf.large}.
Moreover, the reaction term in \eqref{eq:SD-PME.stretched.space.variable} is uniformly bounded, because 
\[
|\sigma{\bar{\source}(\emptyarg,\bar{\solA})} | \leq \LipBoundSource \sigma\mu_-^{-m_0}=\LipBoundSource
\]
due to \eqref{itm:source.Lipschitz.assumption}.
We can now apply known results on uniformly parabolic quasi-linear equations to \eqref{eq:SD-PME.stretched.space.variable}.
In particular, from \cite{lady1968}, Lemma 7.4 of Chapter II, page 119, we obtain the following.
\begin{lemma}\label{lem:Ladyzenskaja.non.deg.improv.osc}
	Suppose statement \textit{(ii)} of \Cref{prop:iterative.scheme} holds.
	Then there exist constants $\theta,\eta\in(0,1)$ depending only on $N$, $c_1$, $c_2$, $m$ and $\LipBoundSource$ such that for any $\bar{R}>0$ with
	\[
	Q^\theta(\bar{R}):=Q(\bar{R},\theta\bar{R}^2)\subseteq \bar{Q}_{n_*}
	\]
	at least one of the following estimates is valid: either
	\begin{equation}\label{eq:Ladyzenskaja.non.deg.improv.osc.case1}
		\essosc_{Q^\theta(\frac{\bar{R}}{4})} \bar{\solA} \leq 2(1-\eta)^{-1} \bar{R}
	\end{equation}
	or
	\begin{equation}\label{eq:Ladyzenskaja.non.deg.improv.osc.case2}
		\essosc_{Q^\theta(\frac{\bar{R}}{4})} \bar{\solA} \leq \eta \essosc_{Q^\theta(\bar{R})}\bar{\solA}.
	\end{equation}
	Moreover, $\theta$ may be chosen arbitrarily small (in particular, $\theta\leq 4^{1-m}$).
\end{lemma}

\begin{remark}
	Historically, the intrinsic scaling introduced in \cite{DiBen1984} was inspired by \Cref{lem:Ladyzenskaja.non.deg.improv.osc} and it leads to such a result for degenerate parabolic equations.
	Indeed, the De Giorgi-type Lemma and the subsequent estimate \eqref{eq:improvement.of.oscillation.parameter.estimate} provide the analogous result, where \eqref{eq:improvement.of.oscillation.parameter.estimate} and the negation of \eqref{eq:oscillation.Really.Large} correspond to  \eqref{eq:Ladyzenskaja.non.deg.improv.osc.case2} and  \eqref{eq:Ladyzenskaja.non.deg.improv.osc.case1}, respectively. 
\end{remark}

We consider an iterative scheme for case \textit{(ii)} of \Cref{prop:iterative.scheme}, where \Cref{lem:Ladyzenskaja.non.deg.improv.osc} replaces the role of \Cref{lem:DeGiorgi-type}.
We take $R_0$ and $\omega_0$ as in \Cref{lem:DeGiorgi-type} and we define decreasing sequences $\{R_k\}_{k=1}^\infty$ and $\{\omega_k\}_{k=1}^\infty$ recursively by
\begin{equation}\label{eq:sequences.R_k.and.omega_k}
	R_{k+1}=\left\{ \begin{array}{rl}
		aR_{k}		&\quad\text{if}\ k< n_*,\\
		\frac{1}{4}R_k	&\quad\text{if}\ k\geq n_*,
	\end{array}\right.
	\quad
	\omega_{k+1}=\left\{ \begin{array}{rl}
		\eta_0 \omega_{k}		&\quad\text{if}\ k< n_*,\\
		\eta \omega_k	&\quad\text{if}\ k\geq n_*,
	\end{array}\right.
\end{equation}
where $a$ and $\eta_0$ are given by \eqref{eq:defining.proportionality.R_n} and \eqref{eq:defining.improvement.of.oscillation.parameter}, respectively.
Recall that $\rho^2=\mu_-^{m-1}\geq 4^{1-m}\omega_{n_*}^{m-1}$ due to \eqref{eq:cond.inf.large}, so
\[
\bar{Q}_{n_*}=B_{\bar{R}_{n_*}}\times(-\omega_{n_*}^{1-m}\rho^2\bar{R}_{n_*}^2,0)\supset B_{\bar{R}_{n_*}}\times(-4^{1-m}\bar{R}_{n_*}^2,0)\supset Q^\theta(\bar{R}_{n_*}),
\]
where we assume that $\theta\leq 4^{1-m}$.
Therefore, if \eqref{eq:Ladyzenskaja.non.deg.improv.osc.case2} holds for $\bar{R}_{n_*}$, then
\[
\essosc_{Q^\theta(\bar{R}_{n_*+1})}\bar{u}\leq\eta\essosc_{\bar{Q}_{n_*}}\bar{u}\leq \eta\omega_{n_*}=\omega_{n_*+1}.
\]
If \eqref{eq:Ladyzenskaja.non.deg.improv.osc.case1} holds for $\bar{R}_k$, $k\geq n_*$, then
\[
\essosc_{Q^\theta(\bar{R}_{{k+1}})}\bar{u}\leq2(1+\eta^{-1}) \bar{R}_{k}\leq \frac{8}{\sigma^{1/2} \rho}(1+\eta^{-1})R_{k+1}\leq C R_{n_*}^{-\frac{m-1+m_0}{2m}}R_{k+1}\leq C R_{k+1}^{1-\frac{m-1+m_0}{2m}},
\]
where we used that $\sigma^{1/2}\rho\geq 2^{-(m-1+m_0)}\omega_{n_*}^{({m-1+m_0})/{2}}\geq 2^{-(m-1+m_0)}R_{n_*}^{({m-1+m_0})/{2m}}$ due to \eqref{eq:cond.inf.large} and \eqref{eq:oscillation.Really.Large} and where we absorbed some factors into the constant $C$.
Let us define $\tilde{\alpha}=\min\left\{ -\log_4(\eta), 1-\frac{m-1+m_0}{2m}, \alpha \right\}$ and update $\alpha=\tilde{\alpha}$.
Combining \eqref{eq:Ladyzenskaja.non.deg.improv.osc.case1} and \eqref{eq:Ladyzenskaja.non.deg.improv.osc.case2}, for $k>n_*$ we now have
\begin{align*}
	\essosc_{Q^\theta(\bar{R}_{k})} \bar{\solA} &\leq \max\left\{CR_k^{\alpha},\omega_k\right\}=\max\left\{CR_k^{\alpha},4^{-(k-n_*)}a^{n_*}\omega_0\right\}\\
	&\leq\max\left\{ CR_k^{\alpha},\omega_0\left(\frac{{R}_k}{{R}_{n_*}}\right)^\alpha\left(\frac{{R}_{n_*}}{{R}_{0}}\right)^\alpha\right\} = C \left( \frac{{R}_k}{{R}_{0}} \right)^\alpha. \label{eq:Ladyzenskaja.non.deg.improv.osc.max}
\end{align*}
Transforming $Q^\theta(\bar{R}_{k})$ back to the original coordinates gives $Q(R_k,\theta
\rho^{-2}R_k^2)$, which certainly contains $Q^\theta(R_k)$.
Therefore,
\begin{equation}
	\essosc_{Q^\theta({R}_{k})} {\solA} \leq C(N,c_1,c_2,m,\LipBoundSource,m_0)\left(\frac{{R}_k}{{R}_{0}}\right)^\alpha	\label{eq:Ladyzenskaja.non.deg.oscillation.scales.with.cylinder}
\end{equation}
for all $k>n_*$.
Moreover, observe that $Q^\theta({R}_{k})\subset Q_{\omega_k}(R_k)$ for any $k\in\IN$, so \eqref{eq:oscillation.intrisinc.scales.with.cylinder} implies that \eqref{eq:Ladyzenskaja.non.deg.oscillation.scales.with.cylinder} is valid for all $k\in\IN$.

We collect both cases of \Cref{prop:iterative.scheme} and our conclusion \eqref{eq:Ladyzenskaja.non.deg.oscillation.scales.with.cylinder} in the following result.
\begin{corollary}\label{prop:iterative.scheme.Continuous.version}
	Suppose $R\in(0,R_{\mathrm{max}}]$ and $\omega>0$ such that \eqref{eq:cond.intrinsic.scaled.cylinder.in.domain}, \eqref{eq:cond.oscillation.recursion.hypothesis} and \eqref{eq:oscillation.Really.Large} hold.
	Then at least one of the following estimates holds:
	\begin{align}
		\essosc_{Q_\omega(r)}\solA&\leq C(N,c_1,c_2,m,\LipBoundSource,m_0)\,\omega\left(\frac{r}{R}\right)^{\alpha} & & \text{for all}\ r\in[0,R], \label{eq:oscillation.bounded.r/R} \\
		\essosc_{Q^\theta(r)}{\solA}&\leq C(N,c_1,c_2,m,\LipBoundSource,m_0)\,\left(\frac{r}{R}\right)^\alpha & & \text{for all}\ r\in[0,R], \label{eq:oscillation.bounded.r/R.non-degen}
	\end{align}
	where $\theta$ is given in \Cref{lem:Ladyzenskaja.non.deg.improv.osc} and $\alpha>0$ depends on $N$, $c_1$, $c_2$, $m$, $\LipBoundSource$ and $m_0$.
\end{corollary}
\begin{proof}
	The statements hold trivially for $r=0$, so assume $r>0$.
	Set $R_0=R$, $\omega_0=\omega$, recall that $a$ and $\eta_0$ are defined in \eqref{eq:defining.proportionality.R_n} and \eqref{eq:defining.improvement.of.oscillation.parameter}, respectively, and define $R_n$, $\omega_n$ and $Q_n$ as in \Cref{prop:iterative.scheme}.
	Next, let $r\in (0,R]$ and pick $n\in\IN$ such that $R_{n+1}\leq r\leq R_n$.
	Note that $Q_{\omega}(r)\subseteq Q_{n}$, because $\omega\geq \omega_n$ and $r\leq R_n$, so $\omega^{1-m}r\leq \omega_{n}^{1-m}R_{n}$.
	
	Suppose case \textit{(i)} of \Cref{prop:iterative.scheme} holds, then it follows that
	\[
	\essosc_{Q_{\omega}(r)}\solA\leq \essosc_{Q_{n}}\solA\leq\omega\left(\frac{R_{n}}{R}\right)^\alpha=a^{-\alpha} \omega \left(\frac{R_{n+1}}{R}\right)^\alpha\leq C\,\omega \left(\frac{r}{R}\right)^\alpha,
	\]
	where $C=a^{-\alpha}$.
	
	Suppose case \textit{(ii)} of \Cref{prop:iterative.scheme} holds.
	Define $R_k$ and $\omega_k$ by \eqref{eq:sequences.R_k.and.omega_k}.
	Pick $k\in \IN$ such that $R_{k+1}\leq r< R_{k}$ and note that $Q^\theta(r)\subseteq Q^\theta(R_k)$.
	Then \eqref{eq:Ladyzenskaja.non.deg.oscillation.scales.with.cylinder} implies
	\begin{equation*}
		\essosc_{Q^\theta(r)}{\solA}\leq \essosc_{Q^\theta(R_k)}{\solA} \leq C\,\left(\frac{R_k}{R}\right)^\alpha\leq C\,\left(\frac{R_{k+1}}{R}\right)^\alpha\leq C\,\left(\frac{r}{R}\right)^\alpha,
	\end{equation*}
	where we absorbed the factor $\max\left\{ a^{-\alpha},4^{\alpha} \right\}$ into the constant $C$.
\end{proof}

Now we are in a position to prove the main result.
\begin{proof}[Proof of \Cref{thm:holder}]
	Let $K\subset\Omega_T$ be compact and let $(x_0,t_0),(x_1,t_1)\in K$.
	Assume $t_0\neq t_1$ and suppose, without loss of generality, that $t_0>t_1$.
	Let $R>0$ and consider the cylinder
	\[
	Q=(x_0,t_0)+Q(R,R^2).
	\]
	Let us pick
	\[
	2R=R_{\mathrm{max}}\cdot d(K;\Gamma)
	\]
	so that $Q\subset (x_0,t_0)+Q(2R,(2R)^2)\subseteq\Omega_T$ and $R\in(0,R_{\mathrm{max}}]$. 
	We consider the following two cases.
	
	$\bullet$ Suppose $(x_1,t_1)\notin Q$, then
	\[
	\abs{x_0-x_1}+\abs{t_0-t_1}^{\frac{1}{2}}\geq R,
	\]
	hence
	\begin{align*}
		\abs{\solA(x_0,t_0)-\solA(x_1,t_1)}	
		&\leq 2\leq 2 \left(\frac{\abs{x_0-x_1}+\abs{t_0-t_1}^{\frac{1}{2}}}{R}\right) \leq 2 \left(\frac{\abs{x_0-x_1}+\abs{t_0-t_1}^{\frac{1}{2}}}{R}\right)^\alpha\\
		&\leq C\left(\frac{\abs{x_0-x_1}+\abs{t_0-t_1}^{\frac{1}{2}}}{d(K;\Gamma)}\right)^\alpha,
	\end{align*}
	where $C={2} / {R_{\mathrm{max}}}$.
	
	$\bullet$ Suppose $(x_1,t_1)\in Q$.
	Set $\omega=1$, then, by the choice of $R$, \eqref{eq:cond.intrinsic.scaled.cylinder.in.domain} is satisfied for the cylinder $(x_0,t_0)+Q_\omega(R)=Q$.
	Moreover, \eqref{eq:cond.oscillation.recursion.hypothesis} and \eqref{eq:oscillation.Really.Large} hold trivially, so the hypotheses of \Cref{prop:iterative.scheme.Continuous.version} are satisfied.
	Suppose \eqref{eq:oscillation.bounded.r/R} holds.
	Define $r\in(0,R]$ by
	\[
	r=\max\left\{\abs{x_0-x_1},\abs{t_0-t_1}^{\frac{1}{2}}\right\},
	\]
	then $(x_1,t_1)\in (x_0,t_0)+Q_\omega(r)\subseteq Q$ and \eqref{eq:oscillation.bounded.r/R} implies that
	\begin{align*}
		\abs{\solA(x_0,t_0)-\solA(x_1,t_1)}&\leq \essosc_{Q_{\omega}(r)} \solA\leq C\left(\frac{r}{R}\right)^\alpha\leq C\left(\frac{\abs{x_0-x_1}+\abs{t_0-t_1}^{\frac{1}{2}}}{d(K;\Gamma)}\right)^{\alpha}.
	\end{align*}
	Suppose \eqref{eq:oscillation.bounded.r/R.non-degen} holds.
	If $(x_1,t_1)\notin Q^\theta(R)$, then 
	\[
	\abs{x_0-x_1}+\theta^{-\frac{1}{2}}\abs{t_0-t_1}^\frac{1}{2}\geq R,
	\]
	hence, as before,
	\begin{align*}
		\abs{\solA(x_0,t_0)-\solA(x_1,t_1)}	&\leq 2\leq 2 \left(\frac{\abs{x_0-x_1}+\theta^{-\frac{1}{2}}\abs{t_0-t_1}^{\frac{1}{2}}}{R}\right)\\
		&\leq 2\theta^{-\frac{1}{2}}\left(\frac{\abs{x_0-x_1}+\abs{t_0-t_1}^{\frac{1}{2}}}{d(K;\Gamma)}\right)^{\alpha}.
	\end{align*}
	Suppose $(x_1,t_1)\in Q^\theta(R)$ and define $r\in(0,R]$ by
	\[
	r=\max\left\{\abs{x_0-x_1},\theta^{-\frac{1}{2}}\abs{t_0-t_1}^{\frac{1}{2}}\right\},
	\]
	then $(x_1,t_1)\in (x_0,t_0)+Q^\theta(r)\subseteq Q$ and \eqref{eq:oscillation.bounded.r/R.non-degen} implies
	\begin{align*}
		\abs{\solA(x_0,t_0)-\solA(x_1,t_1)}&\leq \essosc_{Q^\theta(r)} \solA\leq C\left(\frac{r}{R}\right)^\alpha\leq C\left(\frac{\abs{x_0-x_1}+\abs{t_0-t_1}^{\frac{1}{2}}}{d(K;\Gamma)}\right)^{\alpha},
	\end{align*}
	where we absorbed $\theta^{-\frac{\alpha}{2}}$ into $C$.
	
	Finally, suppose $t_0=t_1$.
	Pick any compact set $K'\supset K$ such that $d(K';\Gamma)\geq \frac{1}{2}d(K;\Gamma)$ and such that $\{x_0,x_1\}\times[t_0,t_0+\tau]\subset K'$ for some $\tau>0$.
	Pick a sequence $\{\tau_n\}_{n=1}^\infty$ in $(t_0,t_0+\tau]$ with $\tau_n\to t_0$.
	\Cref{thm:holder} holds for each pair $(x,t)$ and $(y,s)$, $t\neq s$, in the compact set $K'$, hence $\lim_{n\to\infty}\solA(x_i,\tau_n)=\solA(x_i,t_0)$, $i=0,1$.
	In the limit $n\to\infty$ we obtain \eqref{eq:holder} for $(x_0,t_0), (x_1,t_0)\in K$, where we absorbed the factor $2^\alpha$ into $C$.
\end{proof}

\section{Interior integral estimates and technical lemma's}\label{sect:Int.est.and.Aux}
In this section we show auxiliary lemma's and discuss known technical results we use in the proof of \Cref{lem:DeGiorgi-type} in \Cref{sect:Proof.of.DeGiorgiLemma}.
First, we prove a chain rule for the time derivative satisfied by any solution of \eqref{eq:SD-PME} in the sense of \Cref{def:local.solution}.
Next, we show two interior energy estimates and an interior logarithmic estimate involving truncations of the solution, where we use the aforementioned chain rule.
Finally, we discuss an additional functional space with a related estimate, a Poincar\'e type inequality incorporating truncated functions and a technical lemma on the convergence of certain sequences.

\subsection{Chain rule}
We prove a chain rule for the term involving the time derivative in \eqref{eq:SD-PME}.
The method we use is similar to the one used to prove Proposition 3.10 in \cite{HissMull-Sonn22}.
In the cited reference the initial data provides a key bound to pass to the limit, but here we can use the bound $\solA<1$ instead.
In this subsection we work with the original equation \eqref{eq:SD-PME}, i.e.\ we have not performed the translation of the axes mentioned in the third paragraph of \Cref{sect:Geometry.Holder}.
\begin{lemma}[Chain rule]\label{lem:chainrule_local_solution}
	Let $\solA:\Omega_T\to[0,1)$ be a measurable function such that $\solA \in W^{1,2}_\loc (0,T;H^{-1}(K))$ and $\nonlin(\solA) \in L^2_\loc(0,T;H^1(K))$ for any compact subset $K\subset\Omega$.
	Suppose $\psi:[0,\infty)\to \IR$ is a continuous piece-wise continuously differentiable function with bounded derivative and let $\zeta\in C^\infty_c(\Omega\times\IR)$ be non-negative.
	Then the mapping $t\mapsto \inprod{\Psi(\solA(t))}{\zeta^2(t)}$ is absolutely continuous on any compact subinterval of $(0,T)$ with
	\[
	\frac{\md}{\md t}\inprod{\Psi(\solA)}{\zeta^2}=\pinprod{\solA_t}{\psi(\nonlin(\solA))\zeta^2}+2\inprod{\Psi(\solA)}{\zeta\zeta_t}
	\]
	a.e.\ in $(0,T)$, where $\Psi(z):=\int_{l}^z\psi(\nonlin(\tilde{z}))\md \tilde{z}$, for some ${l}\in [0,1]$.
\end{lemma}

\begin{proof}
	Without loss of generality we may assume that $\psi(0)=0$, since the statement in linear with respect to $\psi$.
	Indeed, first suppose $\psi=c$ for some constant $c$, then $\Psi(\solA)=c(\solA-l)$ and \Cref{lem:chainrule_local_solution} is clearly valid.
	In general, if $\psi(0)=c$, then define ${\psi}_0:=\psi-c$, which satisfies $\psi_0(0)=0$.
	If we prove the proposition for $\psi_0$, then adding the constant function $c$ to $\psi_0$ shows that the result holds for $\psi$ as well.
	
	First, suppose $\Psi$ is convex, i.e.\ $\psi$ is non-decreasing, then one readily checks that
	\begin{equation}\label{eq:transformed.function.estimates}
		\psi(\nonlin(z_2))(z_1-z_2)\leq\Psi(z_1)-\Psi(z_2)\leq \psi(\nonlin(z_1))(z_1-z_2)
	\end{equation} 
	for all $z_1,z_2\in [0,1)$.
	In particular, for $z_2=0$ we have that $0\leq \Psi(z)\leq \psi(\nonlin(z))z$ for any $z\in[0,1)$.
	Moreover, $\psi'$ is bounded, so $\psi(\nonlin(\solA))\in L^2_\loc(0,T;H^1_\loc(\Omega))$ and therefore $\Psi(\solA)\in L^2_\loc(\Omega_T)$.
	
	Let $[t_1,t_2]\subset (0,T)$ be a compact subinterval.
	Given $0<h<t_1$, let $\solA^h$ denote the \emph{backward Steklov average}, that is, using the Bochner integral we define
	\[
	\solA^h(t):=\frac{1}{h}\int_{t-h}^{t}\solA(s)\md s\quad\text{for}\ t\in[t_1,t_2].
	\]
	From known results on Steklov averaging, see the Appendix in \cite{HissMull-Sonn22}, we have $\solA^h\to \solA$ in $H^1(t_1,t_2;H^{-1}({K}))$ as $h\to 0$ and $\solA^h\in H^1(t_1,t_2;L^2(K))$ with $\partial_t\solA^h(t)=\frac{1}{h}(\solA(t)-\solA(t-h)$, for any compact subset $K \subset \Omega$.
	Using the first estimate of \eqref{eq:transformed.function.estimates} and then the second estimate we obtain
	\begin{align*}
		\inprod{\partial_t\solA^h(t)}{\psi(\nonlin(\solA(t)))\zeta^2(t)}
		&=\frac{1}{h}\inprod{\solA(t)-\solA(t-h)}{\psi(\nonlin(\solA(t)))\zeta^2(t)}\\
		&\geq \frac{1}{h}\inprod{\Psi^\star(\solA(t))-\Psi^\star(\solA(t-h))}{\zeta^2(t)}\\
		&=\inprod{\partial_t[\Psi^\star(\solA(t))]^h}{\zeta^2(t)}\\
		&\geq \inprod{\partial_t\solA^h(t)}{\psi(\nonlin(\solA(t-h)))\zeta^2(t)}.
	\end{align*}
	We also note that $\inprod{\partial_t[\Psi^\star(\solA)]^h(t)}{\zeta^2(t)}=\frac{\md}{\md t}\inprod{[\Psi^\star(\solA)]^h(t)}{\zeta^2(t)}-2\inprod{[\Psi^\star(\solA)]^h(t)}{[\zeta\zeta_t](t)}$ by the product rule for the weak derivative in Bochner spaces.
	Integrate the estimates above over $t\in [t_1,t_2]$ and observe that the left- and right-hand side converge to $\int_{t_1}^{t_2}\pinprod{\solA_t}{\psi(\nonlin(\solA))\zeta^2}$ as $h\to 0$, since
	$\psi(\nonlin(\solA))\zeta^2\in L^2(t_1,t_2;H^1(\Omega))$ and 
	\[
	\int_{t_1}^{t_2}\norm{\psi(\nonlin(\solA(t-h)))\zeta^2(t)-\psi(\nonlin(\solA(t)))\zeta^2(t)}_{H^1(\Omega)}\md t\to 0
	\]
	as $h\to 0$.
	It follows that
	\begin{align*}
		\left[\inprod{\Psi^\star(\solA(t))}{\zeta^2(t)}\right]_{t_1}^{t_2}=\int_{t_1}^{t_2}\Bigl(\pinprod{\partial_t\solA}{\psi(\nonlin(\solA))\zeta^2}-2\inprod{[\Psi^\star(\solA)]}{\zeta\zeta_t}\Bigr),
	\end{align*}
	where we used that $[\Psi^\star(\solA)]^h\to \Psi^\star(\solA)$ in $L^2_\loc(\Omega_T)$ as $h\to 0$.
	This proves \Cref{lem:chainrule_local_solution} provided that $\Psi$ is convex.
	
	Finally, suppose $\Psi$ is not convex.
	Then, define $\theta(\zeta)=\int_0^\zeta(\Psi'')_-$ and $\Psi_0(\zeta)=\int_0^\zeta\theta$ for $\zeta\in\IR$, then $\Psi_0''=(\Psi'')_-\geq 0$ and $[\Psi+\Psi_0]''=(\Psi'')_+\geq 0$, so both functions are convex.
	Clearly $\Psi_0$ has the same regularity as $\Psi$ and $\Psi'_0(0)=0$, so $\Psi_0$ satisfies the hypotheses.
	It follows that we can apply the previous arguments to $\Psi_0$ and $\Psi+\Psi_0$.
	Observe that the statements of \Cref{lem:chainrule_local_solution} are linear with respect to $\Psi$ to complete the proof.
\end{proof}

\subsection{Interior integral estimates}\label{subsect:Int.int.estimates}
In this subsection we show that local solutions of \eqref{eq:SD-PME} satisfy three interior integral estimates.
Our arguments rely on the chain rule proved in \Cref{lem:chainrule_local_solution}.

Let $\solA$ be a local solution of \eqref{eq:SD-PME}, let $(x_0,t_0)\in \Omega_T$ and perform the translation of the axes mentioned in the third paragraph of \Cref{sect:Geometry.Holder} to set $(x_0,t_0)=0$.
We assume that $\omega>0$ and $R\in(0,1]$ such that \eqref{eq:cond.intrinsic.scaled.cylinder.in.domain} holds.
Further, we use the following notation:
\[
{\bar{t}_0} := -\omega^{m-1}R^2,\quad {\solA_{(l)}} := \max\left\{ \solA , l \right\} \quad \text{and} \quad {\solA^{(l)}}:=\min\left\{ \solA , l \right\},
\]
given $l\in\IR$.
We also write $\solA_{+}=\max\left\{ \solA,0 \right\}$ and $\solA_{-}=(-\solA)_{+}$.
From now on $\zeta$ will always denote a cut-off function, so in particular
\[
0\leq\zeta\leq 1.
\]
Observe that for $l>0$ we have that ${\solA_{(l)}}\in L^2_\loc(0,T;H^1_\loc(\Omega))$.
Indeed, ${\solA_{(l)}}=\beta(\nonlin(\solA_{(l)}))$, where $\beta:=\nonlin^{-1}$, and $\beta$ restricted to $[\nonlin(l),\infty)$ is a continuously differentiable function with a bounded derivative.
By the same argument, $(\solA-l)_+ \in L^2_\loc(0,T;H^1_\loc(\Omega))$.

\begin{proposition}[Interior energy estimate - lower truncation]\label{prop:interior.energy.estimate.lower.trunc}
	Let $k\geq l>0$ and $\zeta\in C^\infty_c({B_R}\times({\bar{t}_0},\infty))$.
	Then we have estimate
	\begin{equation*}
		\begin{aligned}
			&\norm{({\solA_{(l)}}-{k})_{-}\zeta}_{L^\infty({\bar{t}_0},{0};L^2({B_R}))}^2+c_1l^{m-1}\norm{\nabla({\solA_{(l)}}-{k})_{-}\zeta}_{L^2({Q_\omega(R)})}^2\\
			&\quad\leq C\Biggl( (k-l)(k+l)\iint_{{Q_\omega(R)}}\abs{\zeta_t}\chi_{[\solA<k]}+c_2(k-l)^2k^{m-1}\iint_{{Q_\omega(R)}}\abs{\nabla\zeta}^2\chi_{[l<\solA< k]} \\
			&\qquad\qquad\qquad\qquad\qquad + (k-l) \frac{l^m}{m} \iint_{Q_\omega(R)}\abs{\Delta\zeta}\chi_{[\solA<l]}+\LipBoundSource l^{-m_0} (k-l) \iint_{Q_\omega(R)}\chi_{[\solA<k]} \Biggr)
		\end{aligned}
	\end{equation*}
	for some constant $C\geq 0$.
\end{proposition}
\begin{remark}\label{rem:gradient.interchanging.with.cut-off.function}
	The second term in the estimate is written ambiguously in view of the placement of $\zeta$; it could either mean the norm of $(\nabla({\solA_{(l)}}-{k})_{-})\zeta$ or $\nabla(({\solA_{(l)}}-{k})_{-}\zeta)$.
	This is on purpose, since the estimate holds for both interpretations.
	Indeed, the second case yields the extra term $c_1l^{m-1}\norm{({\solA_{(l)}}-k)_-\nabla\zeta}^2_{L^2({Q_\omega(R)})}$, which is estimated by $c_2k^{m-1}(k-l)^2\iint_{{Q_\omega(R)}}\abs{\nabla\zeta}^2\chi_{[\solA<k]}$.
	Therefore it can be absorbed by the second term on the right-hand side.
\end{remark}
\begin{proof}
	We may assume without loss of generality that $l\leq \mu_+$, because otherwise ${(\solA_{(l)}-k)_-}$ vanishes and the estimate certainly holds.
	We may also assume that $k\leq \mu_+$, because the left-hand side of the estimate does not change if $k>\mu_+$ compared to $k=\mu_+$ while the right-hand side increases.
	
	Observe that 
	\[
	\testA=- {(\solA_{(l)}-k)_-} \zeta^2
	\] 
	is an admissible test function in \eqref{eq:SD-PME.local.solution.including.time-derivative.id}.
	Indeed, define $\psi:[0,\infty)\to\IR$ by
	\[
	\psi(\solB)=-\left(\max\left\{ \beta(\solB),l \right\} -k \right)_-
	\]
	and observe that $\beta:=\nonlin^{-1}$ is continuous on $[0,\infty)$ and $C^1$-regular on $(0,\infty)$.
	We conclude that $\psi$ is continuous and piece-wise $C^1$-regular on $[0,\infty)$ and $\psi'$ is bounded, which implies that $-{(\solA_{(l)}-k)_-}=\psi(\nonlin(\solA))$ is in $L^2_\loc(0,T;H^1_\loc(\Omega))$.
	Applying \eqref{eq:SD-PME.local.solution.including.time-derivative.id} to this test function and integrating over $t\in[{\bar{t}_0},\tau]$ for some fixed $\tau\in[{\bar{t}_0},0]$ gives
	\begin{equation}\label{eq:interior.energy.estimate.lower.trunc.test_function_identity}
		\int_{{\bar{t}_0}}^{\tau}\left[-\pinprod{\solA_t}{{(\solA_{(l)}-k)_-}\zeta^2}-\inprod{\nabla\nonlin(\solA)}{\nabla\left({(\solA_{(l)}-k)_-}\zeta^2\right)}\right]=-\int_{{\bar{t}_0}}^{\tau}\inprod{\source(\emptyarg,\solA)}{{(\solA_{(l)}-k)_-}\zeta^2}.
	\end{equation}
	We study each of the three terms in \eqref{eq:interior.energy.estimate.lower.trunc.test_function_identity} separately.
	
	For the first term we apply \Cref{lem:chainrule_local_solution}, where we observe that
	\begin{align*}
		\Psi^\star(z)&=-\int_{l}^{z}(\max\left\{\tilde{z},l\right\}-k)_-\md\tilde{z}
		=\begin{cases}
			(k-l)(l-z)								&\text{if}\ 0\leq z<{l},\\
			\frac{1}{2}(z-k)_-^2-\frac{1}{2}(k-l)^2	&\text{if}\ z\geq l
		\end{cases}\\
		&=(k-l)(z-l)_-+\tfrac{1}{2}(z_{(l)}-k)_-^2-\tfrac{1}{2}(k-l)^2,
	\end{align*}
	and we note that $\zeta({\bar{t}_0})=0$, to obtain
	\begin{align*}
		-\int_{{\bar{t}_0}}^{\tau}\pinprod{\solA_t}{{(\solA_{(l)}-k)_-}\zeta^2}
		&=\inprod{\Psi^\star(\solA({\tau}))}{\zeta^2({\tau})}-2\int_{{\bar{t}_0}}^{{\tau}}\inprod{\Psi^\star(\solA)}{\zeta_t\zeta}\\
		&=\frac{1}{2}\inprod{{(\solA_{(l)}-k)_-}^2(\tau)}{\zeta^2(\tau)}-\int_{{\bar{t}_0}}^{{\tau}}\inprod{{(\solA_{(l)}-k)_-}^2}{\zeta_t\zeta}\\
		&\quad+(k-l)\inprod{(\solA(\tau)-l)_-}{\zeta^2(\tau)}-2(k-l)\int_{{\bar{t}_0}}^{{\tau}}\inprod{(\solA-l)_-}{\zeta_t\zeta}\\
		&\quad-\frac{1}{2}\inprod{(k-l)^2}{\zeta^2(\tau)}+\int_{{\bar{t}_0}}^{{\tau}}\inprod{(k-l)^2}{\zeta\zeta_t}.
	\end{align*}
	The third term is non-negative and the last two terms cancel.
	Furthermore, we estimate the second and fourth term using ${{(\solA_{(l)}-k)_-}^2}{\zeta_t\zeta}\geq -(k-l)^2\zeta\abs{\zeta_t}\chi_{[\solA<k]}$ and $-2(k-l) (\solA-l)_-\zeta\zeta_t\geq -2l(k-l)\zeta\abs{\zeta_t}\chi_{[\solA<k]}$ to obtain the lower bound
	\begin{equation}
		-\int_{{\bar{t}_0}}^{\tau}\pinprod{\solA_t}{{(\solA_{(l)}-k)_-}\zeta^2}\geq \frac{1}{2}\inprod{{(\solA_{(l)}-k)_-}^2(\tau)}{\zeta^2(\tau)}-(k-l)(k+l)\int_{{\bar{t}_0}}^{{\tau}}\inprod{\zeta}{\abs{\zeta_t}\chi_{[\solA<k]}}.
	\end{equation}
	For the second term in \eqref{eq:interior.energy.estimate.lower.trunc.test_function_identity} we compute
	\begin{align*}
		&-\inprod{\nabla\nonlin(\solA)}{\left(\nabla{(\solA_{(l)}-k)_-}\zeta^2\right)\left\{\chi_{[\solA\geq l]}+\chi_{[\solA< l]}\right\}}\\
		&\quad=\inprod{\nonlin'(\solA)|\nabla{(\solA_{(l)}-k)_-}|^2}{\zeta^2}+\inprod{\nonlin'(\solA){(\solA_{(l)}-k)_-}\nabla{(\solA_{(l)}-k)_-}}{\nabla\zeta^2}\\
		&\qquad-(k-l)\inprod{\nabla\nonlin(\solA^{(l)})}{\nabla\zeta^2}\\
		&\quad\geq \frac{1}{2}\inprod{\nonlin'(\solA)|\nabla{(\solA_{(l)}-k)_-}|^2}{\zeta^2}-2\inprod{\nonlin'(\solA){(\solA_{(l)}-k)_-}^2}{\abs{\nabla\zeta}^2\chi_{[l<\solA<k]}}\\
		&\qquad+(k-l)\inprod{\nonlin(\solA^{(l)})}{\Delta\zeta^2}
	\end{align*}
	a.e.\ in $({\bar{t}_0},{\tau})$, where we used Young's inequality to estimate the second term from below and absorbed one of the resulting terms in the first term.
	Further, we applied integration by parts in space / the definition of the weak derivative to the last term.
	Next, we compute $\Delta\zeta^2=2\abs{\nabla\zeta}^2+2\zeta\Delta\zeta$, of which the first term is non-negative, so we find the lower bound
	\begin{align*}
		&-\inprod{\nabla\nonlin(\solA)}{\left(\nabla{(\solA_{(l)}-k)_-}\zeta^2\right)}\\
		&\quad\geq \frac{1}{2}\inprod{\nonlin'(\solA)|\nabla{(\solA_{(l)}-k)_-}|^2}{\zeta^2}-2\inprod{\nonlin'(\solA){(\solA_{(l)}-k)_-}^2}{\abs{\nabla\zeta}^2\chi_{[l<\solA<k]}}\\
		&\qquad\qquad-2(k-l)\inprod{\nonlin(\solA^{(l)})}{\abs{\Delta\zeta}\zeta}	\\
		&\; \stackrel{\eqref{eq:nonlin.PME-like.degeneracy.integrated},\eqref{eq:extending.PME-like.degeneracy}}{\geq}
		\frac{1}{2}c_1l^{m-1}\inprod{|\nabla{(\solA_{(l)}-k)_-}|^2}{\zeta^2}-2c_2k^{m-1}(k-l)^2\inprod{\abs{\nabla\zeta}^2}{\chi_{[l<\solA<k]}}\\
		&\qquad\qquad-2\frac{c_2}{m}l^{m}(k-l)\inprod{\abs{\Delta\zeta}}{\zeta\chi_{[\solA< l]}}
	\end{align*}
	a.e.\ in $({\bar{t}_0},{\tau})$.
	
	For the last term we simply recall that $\source$ satisfies \eqref{itm:source.Lipschitz.assumption}, so we obtain
	\[
	-\int_{{\bar{t}_0}}^{\tau}\inprod{\source(\emptyarg,\solA)}{{(\solA_{(l)}-k)_-}\zeta^2}\leq \LipBoundSource l^{-m_0}\int_{{\bar{t}_0}}^{\tau}\inprod{({\solA_{(l)}}-k)_{-}}{\zeta^2}\leq \LipBoundSource l^{-m_0} (k-l)\iint_{Q_\tau}\chi_{[\solA<k]},
	\]
	where ${Q_\tau}:={B_R}\times({\bar{t}_0},{\tau})$.
	Combining the three estimates and taking the support of $\zeta$ into account yields
	\begin{equation*}
		\begin{aligned}
			&\int_{B_R\times\{{\tau}\}}({\solA_{(l)}}-{k})_{-}^2\zeta^2+c_1 l^{m-1}\norm{\zeta\nabla({\solA_{(l)}}-{k})_{-}}_{L^2({Q_\tau})}^2\\
			&\quad\leq C(k-l)(k+l)\iint_{{Q_\tau}}\abs{\zeta_t}\chi_{[\solA<k]}+Cc_2(k-l)^2k^{m-1}\iint_{{Q_\tau}}\abs{\nabla\zeta}^2\chi_{[l<\solA< k]} \\
			&\qquad+ C\frac{c_2}{m}(k-l)l^{m}\iint_{Q_\tau}\abs{\Delta\zeta}\chi_{[\solA<l]}+C\LipBoundSource l^{-m_0} (k-l)\iint_{Q_\tau}\chi_{[\solA<k]}.
		\end{aligned}
	\end{equation*}
	Taking the supremum over $\tau\in[-{\bar{t}_0},0]$ yields an estimate for the first term on the left-hand side, where we estimate the right-hand side uniformly with respect to $\tau$ by taking integrals over the larger domain $Q_\omega(R)$ instead, i.e.\ we put $\tau=0$.
	Moreover, putting $\tau=0$ we obtain the same bound for the second term on the left-hand side.
	Adding the two inequalities and absorbing a factor $2$ into $C$ yields the desired estimate.
\end{proof}

Recall that $\mu_+:=\esssup_{{Q_\omega (R)}}\solA$.
\begin{proposition}[Interior energy estimate - upper truncation]\label{prop:interior.energy.estimate.upper.trunc}
	Suppose $k>0$ and $\zeta\in C^\infty_c({B_R}\times({\bar{t}_0},\infty))$.
	Then we have estimate
	\begin{equation}\label{eq:interior.energy.estimate.upper.trunc}
		\begin{aligned}
			&\norm{(\solA-{k})_{+}\zeta}_{L^\infty({\bar{t}_0},{\tau};L^2({B_R}))}^2+c_1k^{m-1}\norm{\nabla(\solA-{k})_{+}\zeta}_{L^2({Q_\omega (R)})}^2\\
			&\quad\leq C\Bigl( (\mu_+
			-k)^2\iint_{{Q_\omega (R)}}\abs{\zeta_t}\chi_{[\solA>k]}+c_2(\mu_+-k)^2\mu_+^{m-1}\iint_{{Q_\omega (R)}}\abs{\nabla\zeta}^2\chi_{[\solA>k]} \\
			&\qquad\qquad+\LipBoundSource k^{-m_0} (\mu_+-k)\iint_{Q_\omega (R)}\chi_{[\solA>k]}\Bigr)
		\end{aligned}
	\end{equation}
	for some constant $C\geq 0$.
\end{proposition}
\begin{remark}
	An analogous statement as in \Cref{rem:gradient.interchanging.with.cut-off.function} holds, that is, the estimate holds for both interpretations of $\nabla(\solA-k)_+\zeta$.
\end{remark}
\begin{proof}
	Without loss of generality we may assume that $k<\mu_+$, because $(\solA-k)_+$ vanishes otherwise and the estimate then certainly holds.
	
	Observe that
	\[
	\eta={(\solA-k)_{+}}\zeta^2
	\]
	is an admissible test function in \eqref{eq:SD-PME.local.solution.including.time-derivative.id} by similar arguments as for the test function in the proof of \Cref{prop:interior.energy.estimate.lower.trunc}, where we now put
	\[
	\psi(\solB)=(\beta(\solB)-k)_+,\quad\text{for}\ \solB\geq 0.
	\]
	Again, $\psi$ is continuous and piece-wise $C^1$-regular on $[0,\infty)$ and $\psi'$ is bounded. 
	We apply \eqref{eq:SD-PME.local.solution.including.time-derivative.id} to this test function, which we integrate over $t\in[{\bar{t}_0},{\tau}]$ for some fixed ${\tau}\in[{\bar{t}_0},0]$.
	We study each of the three terms in the resulting identity separately.
	
	For the first term we apply \Cref{lem:chainrule_local_solution}, where we first observe that
	\[
	\Psi^\star(z)=\int_k^z(\tilde{z}-k)_+\md\tilde{z}=
	\frac{1}{2}(z-k)_+^2
	\]
	and recall that $\zeta({\bar{t}_0})=0$, to obtain
	\[
	\begin{aligned}
		\int_{{\bar{t}_0}}^{{\tau}}\pinprod{\solA_{t}}{{(\solA-k)_{+}}\zeta^2}&=\frac{1}{2}\inprod{(\solA({\tau})-k)_+^2}{\zeta^2}-\int_{\bar{t}_0}^{\tau}\inprod{{(\solA-k)_{+}}^2}{\zeta\zeta_t}.
	\end{aligned}
	\]
	
	For the second term we use Young's inequality to obtain
	\begin{align*}
		&\int_{\bar{t}_0}^{\tau}\inprod{\nabla\nonlin(\solA)}{\nabla\left({(\solA-k)_{+}}\zeta^2\right)}=\int_{\bar{t}_0}^{\tau}\inprod{\nabla\nonlin(\solA)}{\left(\nabla{(\solA-k)_{+}}\right)\zeta^2+{(\solA-k)_{+}}\nabla\zeta^2}\\
		&\geq \frac{1}{2}\nonlin'(k)\iint_{Q_\tau}|\nabla{(\solA-k)_{+}}|^2\zeta^2-2\nonlin'(\mu_+)\iint_{Q_\tau}{(\solA-k)_{+}}^2\abs{\nabla\zeta}^2,
	\end{align*}
	where ${Q_\tau}:={B_R}\times({\bar{t}_0},{\tau})$.
	We use \eqref{eq:extending.PME-like.degeneracy} to estimate $\nonlin'(k)\geq c_1 k^{m-1}$ and $\nonlin'(\mu_+)\leq c_2 \mu_+^{m-1}$.
	
	For the last term we note that $\source$ satisfies \eqref{itm:source.Lipschitz.assumption}, hence as in the proof of \Cref{prop:interior.energy.estimate.lower.trunc} we have the upper bound
	\[
	\int_{{\bar{t}_0}}^{{\tau}}\inprod{\source(\emptyarg,\solA)}{{(\solA-k)_{+}}\zeta^2} \leq \LipBoundSource k^{-m_0} \iint_{Q_\tau}{(\solA-k)_{+}}\leq \LipBoundSource k^{-m_0} (\mu_+-k)\iint_{Q_\tau}\chi_{[\solA>k]}.
	\]
	Combining the three estimates and taking the supremum over ${\tau}\in[{\bar{t}_0},0]$ as in the final step in the proof of \Cref{prop:interior.energy.estimate.lower.trunc} shows that the desired estimate holds.
\end{proof}

Recall that $\mu_+:=\esssup_{{Q_\omega(R)}}\solA$, $\mu_-:=\essinf_{{Q_\omega(R)}}\solA$ and assume $\omega\geq \mu_+-\mu_+$.
\begin{proposition}[Interior logarithmic estimate]\label{prop:int.log.estimate}
	Let $k,l\in\IN$, $l>k$, let $\tau,t\in [{\bar{t}_0},{\tau}]$, $\tau\leq t$ and let $\zeta\in C^\infty_c({B_R})$.
	Then
	\begin{equation*}
		\begin{aligned}
			&(l-k-1)^2\int_{{B_R}\times\{\tau\}}\zeta^2\chi_{[\solA>\mu_-+\omega-\frac{\omega}{2^{l}}]}\\
			&\leq (l-k)^2\int_{{B_R}\times\{t\}}\zeta^2\chi_{[\solA>\mu_-+\omega-\frac{\omega}{2^k}]} + C\, c_2 (l-k)\mu_+^{m-1}\frac{R^2}{\omega^{m-1}}\int_{B_R}\abs{\nabla\zeta}^2 \\
			&\quad +C\, \LipBoundSource \left( \frac{\omega}{2} \right)^{-m_0} 2^{l}\frac{R^2}{\omega^{m}}\abs{{B_R}}
		\end{aligned}
	\end{equation*}
	for some constant $C\geq 0$.
\end{proposition}
\begin{proof}
	Consider the function $\varphi:[0,\mu_+]\to[0,\infty)$ given by
	\[
	\begin{aligned}
		\varphi(z)&=\log^+\left(\frac{\frac{\omega}{2^k}}{\frac{\omega}{2^k}-(z-(\mu_-+\omega-\frac{\omega}{2^k}))_+ +\frac{\omega}{2^l}}\right)\\
		&=\left\{\begin{array}{cl}
			\log\left(\frac{\frac{\omega}{2^k}}{\mu_-+\omega-z+\frac{\omega}{2^l}}\right)&\text{if}\ z\geq\mu_-+\omega-\frac{\omega}{2^k}+\frac{\omega}{2^l},\\
			0&\text{if}\ z<\mu_-+\omega-\frac{\omega}{2^k}+\frac{\omega}{2^l}.
		\end{array}\right.
	\end{aligned}
	\]
	Note that $\varphi$ is a continuous, piece-wise smooth function with bounded derivative.
	We compute
	\begin{align*}
		\varphi'(z)&=\frac{1}{\mu_-+\omega-z+\frac{\omega}{2^l}},\quad\varphi''(z)=\frac{1}{\left(\mu_-+\omega-z+\frac{\omega}{2^l}\right)^2}=(\varphi'(z))^2
	\end{align*}
	for all $z\geq \mu_-+\omega-\frac{\omega}{2^k}+\frac{\omega}{2^l}$, so 
	\[
	(\varphi^2)''=2((\varphi)'^2+\varphi\varphi'')=2(1+\varphi)(\varphi')^2.
	\]
	
	We use the function $\varphi$ to define the test function 
	\[
	\eta=(\varphi^2)'(\solA)\zeta^2,
	\]
	which is admissible in \eqref{eq:SD-PME.local.solution.including.time-derivative.id}, because $\psi:[0,\infty)\to[0,\infty)$,
	\[
	\psi(\solB):=(\varphi^2)'\left(\min\left\{ \beta(\solB) , \mu_+ \right\}\right),
	\]
	is a continuous, piece-wise $C^1$-regular function with bounded derivative.
	In particular,
	\begin{equation}
		\nabla \psi(\nonlin(\solA)) = (\varphi^2)''(\solA)\nabla(\solA-\tfrac{\omega}{2})_+
		=2(1+\varphi(\solA))\varphi'(\solA)^2\nabla(\solA-\tfrac{\omega}{2})_+.\label{eq:int.log.estimate.gradient.of.test_function}
	\end{equation}
	We apply \eqref{eq:SD-PME.local.solution.including.time-derivative.id} to this test function and integrate over $[\tau,t]$ to obtain
	\begin{equation*}
		\int_{{t}}^{\tau}\left[\pinprod{\solA_t}{(\varphi^2)'(\solA)\zeta^2}+\inprod{\nabla\nonlin(\solA)}{\nabla\left((\varphi^2)'(\solA)\zeta^2\right)}\right]=\int_{{t}}^{\tau}\inprod{\source(\emptyarg,\solA)}{(\varphi^2)'(\solA)\zeta^2}.
	\end{equation*}
	We study each of the three terms separately.
	To simplify the notation, we write
	\[
	\varphi(x,t)=\varphi(\solA(x,t)),\quad\varphi'(x,t)=\varphi'(\solA(x,t))\quad\text{and}\quad \varphi''(x,t)=\varphi''(\solA(x,t)).
	\]
	
	For the first term we use \Cref{lem:chainrule_local_solution}.
	First we note that
	\[
	\Psi^\star(\solA)=\int_0^\solA (\varphi^2)'(z)\md z=\varphi^2(\solA)
	\]
	and that $\zeta_t=0$, since $\zeta$ does not depend on $t$.
	We conclude that
	\begin{align*}
		\int_{{t}}^{\tau}\pinprod{\solA_t}{(\varphi^2)'(\solA)\zeta^2}=\inprod{\varphi^2(\solA({\tau}))}{\zeta^2}-\inprod{\varphi^2(\solA({t}))}{\zeta^2}.
	\end{align*}
	
	For the second term we compute
	\begin{align*}
		&\int_{{t}}^{{\tau}}\inprod{\nabla\nonlin(\solA)}{\nabla\left((\varphi^2)'\zeta^2\right)}
		\stackrel{\eqref{eq:int.log.estimate.gradient.of.test_function}}{=}\int_{{t}}^{{\tau}}\inprod{\nabla\nonlin(\solA)}{2(1+\varphi)\varphi'^2\nabla(\solA-\tfrac{\omega}{2})_+\zeta^2+(\varphi^2)'\nabla\zeta^2}\\
		&\quad=2\int_{{t}}^{{\tau}}\inprod{\nonlin'(\solA)(1+\varphi)\varphi'^2\abs{\nabla(\solA-\tfrac{\omega}{2})_+}^2}{\zeta^2}+\int_{{t}}^{{\tau}}\inprod{\nonlin'(\solA)(\varphi^2)'\nabla(\solA-\tfrac{\omega}{2})_+}{\nabla\zeta^2}\\
		&\quad\geq 2\int_{{t}}^{{\tau}}\inprod{\nonlin'(\solA)(1+\varphi)\varphi'^2\abs{\nabla(\solA-\tfrac{\omega}{2})_+}^2}{\zeta^2}-4\int_{{t}}^{{\tau}}\inprod{\nonlin'(\solA)\varphi\varphi'\abs{\nabla(\solA-\tfrac{\omega}{2})_+}}{\zeta\abs{\nabla\zeta}}\\
		&\quad \geq 2\int_{{t}}^{{\tau}}\inprod{\nonlin'(\solA)(1+\varphi)\varphi'^2\abs{\nabla(\solA-\tfrac{\omega}{2})_+}^2}{\zeta^2}-2\int_{{t}}^{{\tau}}\inprod{\nonlin'(\solA)\varphi\varphi'^2\abs{\nabla(\solA-\tfrac{\omega}{2})_+}^2}{\zeta^2}\\
		&\qquad -2\int_{{t}}^{{\tau}}\inprod{\nonlin'(\solA)\varphi}{\abs{\nabla\zeta}^2}\\
		&\quad \geq -2\int_{{t}}^{{\tau}}\inprod{\nonlin'(\solA)\varphi}{\abs{\nabla\zeta}^2} \stackrel{\eqref{eq:extending.PME-like.degeneracy}}{\geq} -2c_2 \mu_+^{m-1}\int_{{t}}^{{\tau}}\inprod{\varphi}{\abs{\nabla\zeta}^2},
	\end{align*}
	where we used Young's inequality to obtain the second estimate.
	
	For the last term we simply recall \eqref{itm:source.Lipschitz.assumption} to obtain estimate
	\begin{equation*}
		\int_{{t}}^{{\tau}}\inprod{\source(\emptyarg,\solA)}{(\varphi^2)'\zeta^2}\leq 2 \LipBoundSource \left(\frac{\omega}{2}\right)^{-m_0}\int_{{t}}^{{\tau}}\inprod{\varphi \varphi'}{\zeta^{2}},
	\end{equation*}
	where we used that $\solA\geq \frac{\omega}{2}$ in the support of $\varphi$.
	
	Combining the three estimates, we have that
	\begin{equation*}
		\begin{aligned}
			\int_{{B_R}}\varphi^2(\solA({\tau}))\zeta^2&\leq\int_{{B_R}}\varphi^2(\solA({t}))\zeta^2+2c_2\mu_+^{m-1}\iint_{Q^\tau}\varphi(\solA)\abs{\nabla\zeta}^2+2\LipBoundSource \mu_-^{-m_0} \int_{Q^\tau}{\varphi(\solA)\varphi'(\solA)\zeta^2},
		\end{aligned}
	\end{equation*}
	where ${Q^\tau}:={B_R}\times({t},{\tau})$.
	Note that $\varphi$ and $\varphi'$ are a non-decreasing functions, so evaluating at $\mu_+$ we find the upper bounds
	\begin{gather*}
		\varphi(\solA)\leq\log\left(\frac{\frac{\omega}{2^k}}{\omega-(\mu_+-\mu_-)+\frac{\omega}{2^l}}\right)\leq\log\left(\frac{\frac{\omega}{2^k}}{\frac{\omega}{2^l}}\right)=(l-k)\log2,\\
		\varphi'(\solA)\leq\frac{1}{\omega-(\mu_+-\mu_-)+\frac{\omega}{2^l}}\leq\frac{2^l}{\omega}.
	\end{gather*}
	Similarly, on the set ${B_R}\cap\left[\solA>\mu_-+\omega-\frac{\omega}{2^{l}}\right]$ the function $\varphi(\solA)$ has the lower bound
	\[
	\varphi(\mu_-+\omega-\tfrac{\omega}{2^{l}})=\log\left(\frac{\frac{\omega}{2^k}}{\frac{\omega}{2^{l}}+\frac{\omega}{2^l}}\right)=\log\left(\frac{2^{l-1}}{2^k}\right)=(l-k-1)\log 2.
	\]
	Substituting these estimates into the inequality and dividing both sides by $(\log(2))^2$, we obtain
	\begin{equation*}
		\begin{aligned}
			&(l-k-1)^2\int_{{B_R}\times\{{\tau}\}}\zeta^2\chi_{[\solA>\mu_-+\omega-\frac{\omega}{2^l}]}\\
			&\leq  (l-k)^2\int_{{B_R}\times\{{t}\}}\zeta^2\chi_{[\solA>\mu_-+\omega-\frac{\omega}{2^k}]}+C\, c_2(l-k)\mu_+^{m-1}\iint_{Q^\tau}\abs{\nabla\zeta}^2 \\
			&\quad + C\, \LipBoundSource \left(\frac{\omega}{2}\right)^{-m_0} (l-k)\frac{2^l}{\omega}\abs{{Q^\tau}}.
		\end{aligned}
	\end{equation*}
	Finally, note that $\abs{{Q^\tau}}\leq\frac{R^2}{\omega^{m-1}}\abs{{B_R}}$ and substitute this in the last two terms.
\end{proof}

\subsection{Embeddings of parabolic spaces and technical lemma's}
We recall the definition of the functional spaces that we will use in the sequel.
\begin{definition}[Parabolic spaces]\label{def:parab.space}
	The Banach spaces $V^2(\Omega_T)$ and $V^2_0(\Omega_T)$ are given by
	\begin{align*}
		V^2(\Omega_T)&=L^\infty(0,T;L^2(\Omega))\cap L^2(0,T;H^1(\Omega)),\\
		V^2_0(\Omega_T)&=L^\infty(0,T;L^2(\Omega))\cap L^2(0,T;H^1_0(\Omega)).
	\end{align*}
	Both spaces are equipped with the norm
	\[
	\norm{\solC}_{V^2(\Omega_T)}=\esssup_{0\leq t\leq T}\norm{\solC(t)}_{L^2(\Omega)}+\norm{\nabla\solC}_{L^2(\Omega_T)}
	\]
	where $\solC\in V^2(\Omega_T)$.
\end{definition}
The following result shows that the space $V^2_0(\Omega_T)$ is embedded into $L^2(\Omega_T)$.
\begin{lemma}\label{lem:parab.space.embedding}
	Let $\Omega\subset \IR^N$ be any bounded domain and $0<T<\infty$.
	Then there exists a constant $C\geq 0$ depending only on $N$ such that
	\begin{equation}\label{eq:parab.space.embedding}
		\norm{\solC}_{L^2(\Omega_T)}\leq C\abs{\Omega_T\cap\left[\solC\neq 0\right]}^{\frac{1}{N+2}}\norm{\solC}_{V^2(\Omega_T)}
	\end{equation}
	for all $\solC\in V^2_0(\Omega_T)$.
\end{lemma}
\begin{proof}
	This inequality is given in \cite{lady1968}, Equation (3.7) on page 76.
\end{proof}

Let us mention the following Poincar\'e type inequality that is concerned with truncated functions. Actually, this is a lemma due to De Giorgi, see Lemma 2.3 in \cite{DiBenedetto1982}.
\begin{lemma}\label{lem:Poincare-type.ineq.tracking.levelsets}
	Let $\solC\in W^{1,1}(B_R)$ and let $l,k$ be any reals such that $l>k$. 
	Then there exists a constant $C\geq 0$ depending only on dimension $N$ such that
	\begin{align*}
		(l-k)\abs{B_R\cap\left[\solC>l\right]}^{1-\frac{1}{N}}\leq C\frac{R^N}{\abs{B_R\cap\left[\solC\leq k\right]}}\int_{B_R\cap\left[l>\solC\geq k\right]}\abs{\nabla \solC}\md x.
	\end{align*}
\end{lemma}
\begin{proof}
	This statement is proven in \cite{ladyUral1968}, Lemma 3.5 on p.\ 55.
\end{proof}

Finally, we have the following lemma on fast geometric convergence.

\begin{lemma}\label{lem:fast.geometric.convergence}
	Assume that $\{y_n\}_{n=1}^\infty$ is a non-negative sequence such that
	\begin{align*}
		y_{n+1}\leq C\, b^ny_n^{1+{a}}
	\end{align*}
	for all $n\in\mathbb{N}$ for some constants $C>0$, $b>1$ and ${a}\in(0,1)$. If $y_0\leq \theta$, then
	\begin{align*}
		y_n\leq \theta\, b^{-n{a}^{-1}},
	\end{align*}
	where $\theta:= C^{-{a}^{-1}}b^{-{a}^ {-2}}$.
\end{lemma}
\begin{proof}
	A simple proof by induction can by found in \cite{ladyUral1968}, Lemma 4.7 on page 66.
\end{proof}
\section{Proof of the De Giorgi-type Lemma} \label{sect:Proof.of.DeGiorgiLemma}
We prove each of the two alternatives in \Cref{lem:DeGiorgi-type} separately.
\subsection{The first alternative}
We assume the hypotheses of \Cref{lem:DeGiorgi-type} are satisfied, i.e.\ $\omega>0$ and $R\in(0,R_{\mathrm{max}}]$ such that \eqref{eq:cond.intrinsic.scaled.cylinder.in.domain}, \eqref{eq:cond.oscillation.recursion.hypothesis}, \eqref{eq:cond.inf.small} and \eqref{eq:oscillation.Really.Large} hold.
Note that \eqref{eq:cond.oscillation.recursion.hypothesis} and \eqref{eq:cond.inf.small} imply $\mu_+\leq \frac{5}{4}\omega$.
\begin{proposition}\label{prop:DeGiorgi.first.alternative}
	There exists a $\nu_0\in (0,1)$ depending on $N$, $c_1$, $c_2$, $m$ and $\LipBoundSource$ such that, if \eqref{eq:DeGiorgi-type.alternative.I} holds,	then $\solA> \mu_- + \frac{\omega}{2}$ a.e.\ in $Q_\omega(\frac{R}{2})$.
\end{proposition}
\begin{proof}
	We define sequences
	\[
	R_n=\frac{R}{2}+\frac{R}{2^{n+1}},\quad k_n=\mu_-+\frac{\omega}{4}+\frac{\omega}{2^{n+2}}
	\]
	and construct the family of nested and shrinking cylinders ${Q}_n=Q_\omega(R_n)$.
	Let $\zeta_n$ be a smooth cut-off function corresponding to the inclusion ${Q}_{n+1}\subset {Q}_{n}$, i.e.\ $0\leq\zeta\leq 1$, $\zeta_n$ vanishes on the parabolic boundary of ${Q}_{n}$, $\zeta_n\equiv 1$ in ${Q}_{n+1}$ and
	\[
	\abs{\nabla\zeta_n}\leq\frac{2^{n+2}}{R},\quad\abs{\Delta\zeta_n}\leq \frac{2^{2(n+2)}}{R^2},\quad 0\leq \zeta_{n,t}\leq 2^{2(n+2)}\frac{\omega^{m-1}}{R^2}.
	\]
	We use \Cref{prop:interior.energy.estimate.lower.trunc}, substituting $R=R_n$, $k=k_n$ and $l=\mu_-+\frac{\omega}{4}$, and we write $\solA_\omega={\solA_{(l)}}=\max\left\{ \solA,\mu_-+\frac{\omega}{4} \right\}$.
	Moreover, we have the following inequalities:
	\begin{gather*}
		\frac{\omega}{4}\leq l; \quad 
		(k-l)(k+l)=\frac{\omega}{2^{n+2}}(2\mu_-+\frac{\omega}{2}+\frac{\omega}{2^{n+2}})\leq\frac{\omega^2}{2^{n}}; \\
		(k-l)^2k^{m-1} \leq \frac{\omega^2}{2^{2(n+2)}}\omega^{m-1}\leq \frac{\omega^{m+1}}{2^{2n}}; \; (k-l)k^m\leq \frac{\omega}{2^{n+2}}\omega^{m}\leq \frac{\omega^{m+1}}{2^{n}}; \\
		\LipBoundSource l^{-m_0} (k-l) \leq \LipBoundSource 4^{m_0}\omega^{-m_0} \frac{\omega}{2^{n}}.
	\end{gather*}
	After absorbing $4^{m-1}$, $c_1^{-1}$, $c_2$, $m^{-1}$, $\LipBoundSource$ and $4^{m_0}$ into $C$, the interior energy inequality in \Cref{prop:interior.energy.estimate.lower.trunc} reads
	\begin{equation*}
		\begin{aligned}
			&\norm{(\solA_\omega-{k_n})_{-}{\zeta}_n}_{L^\infty({t},{\tau};L^2(B(R_{n})))}^2+\omega^{m-1}\norm{\nabla((\solA_{\omega}-{k_n})_{-}{\zeta}_n)}_{L^2(Q_{n})}^2\\
			&\quad\leq C(c_1,c_2,m,\LipBoundSource,m_0)\Bigl( 2^{n}\frac{\omega^{m+1}}{R^2}\iint_{{Q}_n}\chi_{[\solA<k_n]}+\frac{\omega^{m+1}}{R^2}\iint_{{Q}_n}\chi_{[\solA< k_n]} \\
			&\qquad\qquad\qquad\qquad\qquad +2^{n}\frac{\omega^{m+1}}{R^2}\iint_{{Q}_n}\chi_{[\solA<\mu_-+\frac{\omega}{4}]}+\frac{\omega^{1-m_0}}{2^{n}}\iint_{{Q}_n}\chi_{[\solA<k_n]} \Bigr) \\
			&\quad\leq C(c_1,c_2,m,\LipBoundSource,m_0)\left(2^{n}\frac{\omega^{m+1}}{R^2}\iint_{{Q}_n}\chi_{[\solA<k_n]}+ 2^{n}\frac{\omega^{m+1}}{R^2} \left(\frac{R^2}{\omega^{m+m_0}}\right)\iint_{{Q}_n}\chi_{[\solA<k_n]}\right)\\
			&\quad\leq C(c_1,c_2,m,\LipBoundSource,m_0)\ 2^{n}\frac{\omega^{m+1}}{R^2}\iint_{{Q}_n}\chi_{[\solA<k_n]},
		\end{aligned}
	\end{equation*}
	where the second inequality is obtained by combining the first three terms and the fact that $2^{-n}\leq 1 \leq 2^n$.
	The third inequality is derived by noting that \eqref{eq:oscillation.Really.Large} implies that
	\[
	\frac{R^2}{\omega^{m+m_0}}\leq R^{1-\frac{m_0}{m}} \leq 1
	\]
	due to $0<m_0<m$ by assumption.
	
	We introduce the change of variables 
	\[
	\tau=\omega^{m-1}t,
	\]
	denote by ${\tilde{Q}_n}$ the transformed cylinder $B_{R_n}\times(-R_n^2,0)$ and define the transformed function
	\[
	\solAvar_\omega(x,\tau):=\solA_\omega(x,\omega^{1-m}\tau), \quad \tilde{\zeta}(x,\tau):=\zeta(x,\omega^{1-m}\tau).
	\]
	The above estimate transforms into
	\begin{equation*}
		\begin{aligned}
			&\norm{(\solAvar_\omega-{k_n})_{-}\tilde{\zeta}_n}_{L^\infty(-R_{n}^2,0;L^2(B(R_{n})))}^2+\norm{\nabla((\solAvar_{\omega}-{k_n})_{-}\tilde{\zeta}_n)}_{L^2(\tilde{Q}_{n})}^2\\
			&\quad\leq  C(c_1,c_2,m,\LipBoundSource,m_0)\ 2^{n}\frac{\omega^2}{R^2}\iint_{\tilde{Q}_n}\chi_{[\solAvar<k_n]},
		\end{aligned}
	\end{equation*}
	in other words
	\[
	\norm{(\solAvar_\omega-{k_n})_{-}\tilde{\zeta}_n}_{V^2(\tilde{Q}_{n})}^2\leq  C(c_1,c_2,m,\LipBoundSource)\ 2^{n}\frac{\omega^2}{R^2}|{\tilde{Q}_{n}\cap[\solAvar<k_n]}|.
	\]
	
	We apply \Cref{lem:parab.space.embedding} to the function $\solC=(\solAvar_\omega-{k_n})_{-}\tilde{\zeta}_n\in V^2_0(\tilde{Q}_n)$ to estimate its $L^2$-norm in terms of its $V^2$-norm to obtain
	\begin{align*}
		&\frac{\omega^2}{2^{2(n+3)}}|{\tilde{Q}_{n+1}\cap[\solAvar<k_{n+1}]}|=(k_{n}-k_{n+1})^2|{\tilde{Q}_{n+1}\cap[\solAvar<k_{n+1}]}|\\
		&\leq \norm{(\solAvar_\omega-{k_n})_{-}\tilde{\zeta}_n}_{L^2(\tilde{Q}_{n})}^2\leq C(N) |\tilde{Q}_{n}\cap[\solAvar<k_n]|^{\frac{2}{N+2}}\norm{(\solAvar_\omega-{k_n})_{-}\tilde{\zeta}_n}_{V^2(\tilde{Q}_{n})}^2\\
		&\leq C(N,c_1,c_2,m,\LipBoundSource,m_0)\ 2^{n}\frac{\omega^2}{R^2}|\tilde{Q}_{n}\cap[\solAvar<k_n]|^{1+\frac{2}{N+2}}.
	\end{align*}
	
	Next, let us set ${y}_n=\frac{|\tilde{Q}_{n}\cap[\solAvar<k_n]|}{|\tilde{Q}_{n}|}$ and note that $|\tilde{Q}_{n}|=\abs{B_1}R_n^{N+2}$, where $B_1$ is the $N$-dimensional unit sphere, so
	\[
	\begin{aligned}
		\frac{|\tilde{Q}_{n}|^{1+\frac{2}{N+2}}}{|\tilde{Q}_{n+1}|}&={\abs{B_1}}^{\frac{2}{N+2}}\frac{(R_n^{N+2})^{1+\frac{2}{N+2}}}{R_{n+1}^{N+2}}=C(N)\ R_n^2\left(\frac{R_n}{R_{n+1}}\right)^{N+2}\\
		&=C(N)\ R_n^2\left(\frac{\frac{1}{2}+\frac{1}{2^{n+2}}}{\frac{1}{2}+\frac{1}{2^{n+3}}}\right)^{N+2}\leq C(N)\ 2^{N+2}R^2.
	\end{aligned}
	\]
	We conclude that
	\[
	y_{n+1}\leq C(N,c_1,c_2,m,\LipBoundSource,m_0)\ 2^{3n} y_{n}^{1+\frac{2}{N+2}}.
	\]
	By \Cref{lem:fast.geometric.convergence}, if
	\[
	y_0\leq  C(N,c_1,c_2,m,\LipBoundSource,m_0)^{-\frac{N+2}{2}}2^{-(N+2)^2},
	\]
	then $y_n\to 0$ as $n\to\infty$.
	Let us pick 
	\begin{equation}\label{eq:def.nu_0}
		\nu_0:= C(N,c_1,c_2,m,\LipBoundSource,m_0)^{-\frac{N+2}{2}}2^{-(N+2)^2},
	\end{equation}
	then the statement is proven, since $y_n\to0$ implies that 
	\[
	\left| {{Q}_n\cap \left[\solA<\mu_-+\frac{\omega}{4}+\frac{\omega}{2^{n+2}}\right]} \right|\to 0
	\]
	as $n\to \infty$, and therefore $\abs{{Q}\left( \frac{R}{2} \right) \cap[\solA\leq\mu_-+\frac{\omega}{2}]}=0$.
\end{proof}

\subsection{The second alternative}
As in the previous subsection we assume that the hypotheses of \Cref{lem:DeGiorgi-type} are satisfied.
Observe that \eqref{eq:DeGiorgi-type.alternative.I} is violated if and only if \eqref{eq:DeGiorgi-type.alternative.II} is satisfied, because
\begin{equation*}
	\abs{Q_\omega(R)\cap\left[\solA<\mu_-+\frac{\omega}{2}\right]}\geq\nu_0\abs{Q_\omega(R)}\\
\end{equation*}
if and only if
\begin{equation*}
	\begin{aligned}
		\abs{Q_\omega(R)\cap\left[\solA\geq\mu_-+\frac{\omega}{2}\right]}&=\abs{Q_\omega(R)}-\abs{Q_\omega(R)\cap\left[\solA<\mu_-+\frac{\omega}{2}\right]}\\
		&\leq(1-\nu_0)\abs{Q_\omega(R)}.
	\end{aligned}
\end{equation*}
This justifies calling the two alternatives in \Cref{lem:DeGiorgi-type} a dichotomy.

We prove the second alternative of \Cref{lem:DeGiorgi-type} in two steps.
In the first step we generalize the condition \eqref{eq:DeGiorgi-type.alternative.II}.
In particular, we show that we may replace the factor $1-\nu_0$ by any $\nu\in(0,1)$ on the right-hand side of \eqref{eq:DeGiorgi-type.alternative.II} at the cost of shrinking the set involved in the left-hand side appropriately.
This is necessary due to the fact that $\nu_0$ has been fixed by \Cref{prop:DeGiorgi.first.alternative}.
Indeed, $\nu_0$ was chosen such that \Cref{lem:fast.geometric.convergence} could be applied.
In the second alternative we do not have this freedom of choice and therefore this additional step is required.
Here, we use the interior logarithmic estimate, that is, \Cref{prop:int.log.estimate}.
In the second step we prove the second alternative in the same manner as in \Cref{prop:DeGiorgi.first.alternative}.

We start with the first step by proving an auxiliary estimate. 
Recall that ${\bar{t}_0}=-\omega^{1-m}R^2$.
\begin{lemma}\label{lem:Alternative.condition.specific.tau}
	Let $\nu_0$ be given by \Cref{prop:DeGiorgi.first.alternative} and suppose \eqref{eq:DeGiorgi-type.alternative.II} holds.
	Then there exists $\tau\in [{\bar{t}_0},{\frac{\nu_0}{2}\bar{t}_0}]$ such that
	\[
	\abs{{B_R}\cap\left[\solA(\tau)>\mu_-+\frac{\omega}{2}\right]}\leq \frac{1-\nu_0}{1-\frac{\nu_0}{2}}\abs{{B_R}}.
	\]
\end{lemma}
\begin{proof}
	Suppose that the inequality does not hold for all $\tau\in[{\bar{t}_0},{\frac{\nu_0}{2}\bar{t}_0}]$, then
	\begin{align*}
		\abs{{Q_\omega(R)}\cap\left[\solA>\mu_-+\frac{\omega}{2}\right]}&\geq \int_{{\bar{t}_0}}^{{\frac{\nu_0}{2}\bar{t}_0}}\abs{{B_R}\cap\left[\solA(\tau)>\mu_-+\frac{\omega}{2}\right]}\md \tau\\
		&>(1-\tfrac{\nu_0}{2})\omega^{-(m-1)}R^2\frac{1-\nu_0}{1-\frac{\nu_0}{2}}\abs{{B_R}}=(1-\nu_0)\abs{{Q_\omega(R)}}.
	\end{align*}
	This inequality implies \eqref{eq:DeGiorgi-type.alternative.I}, since
	\begin{align*}
		&\abs{{Q_\omega(R)}\cap\left[\solA<\mu_-+\frac{\omega}{2}\right]}\leq\abs{{Q_\omega(R)}\cap\left[\solA\leq\mu_-+\frac{\omega}{2}\right]}\\
		&\quad=\abs{{Q_\omega(R)}}-\abs{{Q_\omega(R)}\cap\left[\solA>\mu_-+\frac{\omega}{2}\right]}<\nu_0\abs{{Q_\omega(R)}},
	\end{align*}
	so we have a contradiction, which proves the lemma.
\end{proof}
Next, we aim to extend \Cref{lem:Alternative.condition.specific.tau} to the interval $[{\frac{\nu_0}{2}\bar{t}_0},0]$ instead of only a specific $\tau$ by reducing the size of the set on the left-hand side.
\begin{corollary}\label{cor:Alternative.condition.time-interval}
	Let $\nu_0$ be given by \Cref{prop:DeGiorgi.first.alternative} and suppose \eqref{eq:DeGiorgi-type.alternative.II} holds.
	There exists an integer $n_*$ depending on $N$, $c_1$, $c_2$, $m$ and $\LipBoundSource$ such that
	\[
	\abs{{B_R}\cap\left[\solA(t)\geq\mu_-+\omega-\frac{\omega}{2^{n_*}}\right]}\leq \left(1-\left(\frac{\nu_0}{2}\right)^2\right)\abs{{B_R}}
	\]
	for all $t\in[{\frac{\nu_0}{2}\bar{t}_0},0]$.
\end{corollary}
\begin{proof}
	We use the interior logarithmic estimate, i.e.\ \Cref{prop:int.log.estimate}.
	Let $\lambda\in(0,1)$, to be determined later, consider the ball $B_{(1-\lambda)R}$ and let $\zeta$ be the corresponding cut-off function, i.e.\ $\zeta\in C^\infty_c({B_R})$ such that
	\[
	0\leq\zeta\leq 1,\quad\zeta\equiv 1\ \text{in}\ B_{(1-\lambda)R},\quad\abs{\nabla\zeta}\leq\frac{C}{\lambda R}.
	\]
	Now, let $k,l\in\IN$, $l>k$.
	Then \Cref{prop:int.log.estimate} gives
	\begin{equation*}
		\begin{aligned}
			&(l-k-1)^2\abs{B_{(1-\lambda)R}\cap\left[\solA(t)>\mu_-+\omega-\tfrac{\omega}{2^l}\right]}\\
			&\leq (l-k)^2\abs{{B_R}\cap\left[\solA(\tau)>\mu_-+\omega-\tfrac{\omega}{2^k}\right]}+ C\left( \frac{(l-k)}{\lambda^2}\abs{{B_R}}+ \LipBoundSource \left(\frac{\omega}{2}\right)^{-m_0} 2^{l}\frac{R^2}{\omega^m}\abs{{B_R}}\right),
		\end{aligned}
	\end{equation*}
	where we used that $\mu_+^{m-1}\frac{1}{\omega^{m-1}}\leq C$ by the assumption $\mu_+\leq\frac{5}{4}\omega$.
	Moreover, $\mu_-+\omega-\frac{\omega}{2^k}\geq \mu_-+\frac{\omega}{2}$, so \Cref{lem:Alternative.condition.specific.tau} implies that
	\[
	\begin{aligned}
		&(l-k-1)^2\abs{B_{(1-\lambda)R}\cap\left[\solA(t)>\mu_-+\omega-\tfrac{\omega}{2^l}\right]}\\
		&\leq (l-k)^2\frac{1-\nu_0}{1-\frac{\nu_0}{2}}\abs{{B_R}}+C\ \frac{(l-k)}{\lambda^2}\abs{{B_R}}+C(\LipBoundSource,m_0)\, 2^{l+1}\frac{R^2}{\omega^{m+m_0}}\abs{{B_R}}.
	\end{aligned}
	\]
	Next, note that $\abs{{B_R}\backslash B_{(1-\lambda)R}}=\abs{{B_R}}-\abs{B_{(1-\lambda)R}}\leq\lambda N\abs{{B_R}}$, because the volume of an $N$-dimensional spherical shell of width $\lambda$ is proportional to
	\begin{align*}
		R^N-(1-\lambda)^NR^N&=(1-b^N)R^N=(1-b)(1+b+b^2+\ldots+b^{N_1})R^N\\
		&\leq(1-b) NR^N=\lambda NR^N,\quad b:=1-\lambda.
	\end{align*}
	We apply this to obtain the estimate
	\begin{equation*}
		\begin{aligned}
			&\abs{{B_R}\cap\left[\solA(t)>\mu_-+\omega-\tfrac{\omega}{2^l}\right]} \leq \\
			& \left(\left(\frac{l-k}{l-k-1}\right)^2\frac{1-\nu_0}{1-\frac{\nu_0}{2}}+C \frac{(l-k)}{(l-k-1)^2}\frac{1}{\lambda^2}+  \frac{C(\LipBoundSource,m_0) \, 2^{l+1}}{(l-k-1)^2}\frac{R^2}{\omega^{m+m_0}}+C N\lambda\right)\abs{{B_R}}.
		\end{aligned}
	\end{equation*}
	
	Let us set $k=1$ and pick $\lambda$ and $l$ in an appropriate manner to obtain the desired estimate.
	First, we chose $\lambda\in(0,1)$ small enough such that
	\[
	C N\lambda<\frac{1}{4}\nu_0^2.
	\]
	Then, we pick $l$ large enough such that
	\[
	\left(\frac{l-1}{l-2}\right)^2\leq\left(1-\frac{\nu_0}{2}\right)(1+\nu_0),
	\]
	which is possible, because $\left(1-\frac{\nu_0}{2}\right)(1+\nu_0)>1$ and $\left(\frac{l-1}{l-2}\right)^2\to 1$ as $l\to\infty$.
	So we can bound the first term by
	\[
	(1-\nu_0^2)\abs{{B_R}}=\left(1-\left(\frac{\nu_0}{2}\right)^2\right)\abs{{B_R}}-\frac{3}{4}\nu_0^2\abs{{B_R}}.
	\]
	If needed, we pick $l$ larger such that
	\[
	C \frac{(l-1)}{(l-2)^2}\frac{1}{\lambda^2}\leq\frac{1}{4}\nu_0^2.
	\]
	Finally, we have
	\[
	C(\LipBoundSource,m_0) \, \frac{2^{l+1}}{(l-2)^2}\frac{R^2}{\omega^{m+m_0}}\leq \frac{1}{4}\nu_0^2,
	\]
	because we assume that $R\leq R_\mathrm{max}$.
	Indeed, recall that \eqref{eq:oscillation.Really.Large} is satisfied by assumption, that is, $\omega^{-m-m_0} R^2\leq R^{(m-m_0)/m}$, so we set
	\begin{equation}\label{eq:define.Rmax}
		R_{\mathrm{max}}=\left(\frac{\nu_0^2}{4C}\frac{(l-2)^2}{2^{l}}\right)^{\frac{m}{m-m_0}}
	\end{equation}
	to ensure that this bound holds.
	Set $n_*=l$ to finish the proof.	
\end{proof}

The conclusion of the first step in the proof of the second alternative of \Cref{lem:DeGiorgi-type}, i.e.\ the desired generalization of the condition  \eqref{eq:DeGiorgi-type.alternative.II}, is given by the following lemma.
Recall that
\[
Q^{\nu_0}_\omega(R):=Q\left( \tfrac{\nu_0}{2}\omega^{1-m}R^2 , R \right).
\]
\begin{lemma}\label{lem:alternative.condition.refined}
	Let $\nu_0$ be given by \Cref{prop:DeGiorgi.first.alternative} and suppose \eqref{eq:DeGiorgi-type.alternative.II} holds.
	For any $\nu\in(0,1)$ there exists a $n_0>n_*$ such that
	\begin{equation*}
		\abs{\Qrho\cap \left[\solA>\mu_-+\omega-\tfrac{1}{2^{n_0}}\omega\right]} < \nu\abs{\Qrho}.
	\end{equation*}
\end{lemma}
Before we prove \Cref{lem:alternative.condition.refined}, we need the following auxiliary lemma that allows us to estimate the gradient of the truncated solution.
It is an immediate consequence of the interior energy inequality given in \Cref{prop:interior.energy.estimate.upper.trunc}.
Here, the second inclusion of \eqref{eq:cond.intrinsic.scaled.cylinder.in.domain}, that is, $Q_\omega(2R)\subseteq\Omega_T$, is a key ingredient.

\begin{lemma}\label{lem:Gradient.Estimate.via.Energy.Ineq}
	Let $\nu_0\in(0,1)$, then we have the estimate
	\[
	\begin{aligned}
		\norm{\nabla(\solA-(\mu_-+(1-\tfrac{1}{2^n})\omega))_+}_{L^2(\Qrho}^2
		\leq C(N,c_1,c_2,m,\LipBoundSource,m_0)\ \frac{\omega^2}{2^{2n}R^2} \abs{Q^{\nu_0}_\omega(R)}.
	\end{aligned}
	\]
\end{lemma}
\begin{proof}
	Fix $n\in\IN$.
	Without loss of generality we may assume that $\omega$ is small enough such that 
	\[
	k:=\mu_-+\left( 1-\frac{1}{2^n} \right) \omega<\mu_+.
	\] 
	Otherwise, the left-hand side of the estimate vanishes and the lemma holds trivially.
	It follows that
	\[
	(\mu_+-k)^2=(\mu_+-\mu_--(1-\frac{1}{2^{n}})\omega)^2\leq \frac{\omega^2}{2^{2n}}.
	\]
	We also have that $\frac{1}{2}\omega\leq k$ and $\mu_+\leq \frac{5}{4}\omega\leq 2\omega$.
	
	Let $\zeta$ be a smooth cut-off function corresponding to the inclusion $\Qrho\subset Q^{\nu_0}_\omega(2R)$, then
	\[
	\abs{\nabla\zeta}\leq\frac{4}{R},\quad 0\leq\zeta_t\leq\frac{8\omega^{m-1}}{\nu_0R^2}
	\]
	and $\zeta$ vanishes on the parabolic boundary of $Q^{\nu_0}_\omega(R)$.
	We use \Cref{prop:interior.energy.estimate.upper.trunc}, substituting $\omega^{1-m}$ by $\frac{\nu_0}{2}\omega^{1-m}$ and $R$ by $2R$
	to obtain an estimate of the $L^2$-norm of the gradient, namely
	\begin{align*}
		&c_1\left(\frac{\omega}{2}\right)^{m-1}\norm{\nabla(\solA-k)_+}_{L^2(\Qrho)}^2 \\
		&\quad \leq C \left(\frac{\omega^{m+1}}{\nu_02^{2n}R^2}+c_2\frac{(2\omega)^{m+1}}{2^{2n}R^2}+\LipBoundSource \left(\frac{\omega}{2}\right)^{-m_0} \frac{\omega}{2^{n}}\right) \iint_{Q^{\nu_0}_{\omega}(2R)}\chi_{[\solA>k]}.
	\end{align*}
	Multiplying the left- and right-hand side by $\omega^{1-m}$ and absorbing the constants into $C$ we obtain
	\[
	\norm{\nabla(\solA-k)_+}_{L^2(\Qrho)}^2\leq C(N,c_1,c_2,m,\LipBoundSource,m_0)\ \left(1+\frac{R^2}{\omega^{m+m_0}}\right)\frac{\omega^2}{2^{2n}R^2} \abs{Q^{\nu_0}_\omega(R)}.
	\]
	By \eqref{eq:oscillation.Really.Large} we know that $\omega^{-m-m_0}R^2\leq R^{1-\frac{m_0}{m}}\leq 1$, hence the estimate is proved.
\end{proof}

\begin{proof}[Proof of \Cref{lem:alternative.condition.refined}]
	We apply \Cref{lem:Poincare-type.ineq.tracking.levelsets} with $\solC=\solA(t)$, $t\in (-\tfrac{\nu_0}{2}\omega^{1-m}R^2,0)$,
	\[
	l=\mu_-+(1-\tfrac{1}{2^{n+1}})\omega \quad \text{and} \quad k=\mu_-+(1-\tfrac{1}{2^n})\omega,
	\]
	where $n>n_*$ is fixed and $n_*$ is given by \Cref{cor:Alternative.condition.time-interval}.
	We multiply the left- and right-hand side of the estimate by $\abs{{B_R}\cap[\solC>l]}^{\frac{1}{N}}\leq C(N) R$ and we bound the integral on the right-hand side of the resulting estimate by
	\begin{equation}\label{eq:alternative.condition.refined.Holder.ineq.result}
		\int_{{B_R}}\abs{\nabla(\solA(t)-k)_+}\chi_{[k\leq\solC<l]} \leq \left(\int_{{B_R}}\abs{\nabla(\solA(t)-k)_+}^2\chi_{[k\leq\solC<l]}\right)^{\frac{1}{2}}\abs{{B_R}\cap\left[k\leq\solA(t)\leq l\right]}^{\frac{1}{2}},
	\end{equation}
	which holds by H\"older's inequality.
	By \Cref{cor:Alternative.condition.time-interval} we know that
	\[
	\abs{{B_R}\cap\left[\solA(t)\leq k\right]}\geq C(N)\ \left(\tfrac{\nu_0}{2}\right)^2 R^N,
	\]
	hence $R^{N+1}{\abs{{B_R}\cap[\solA(t)\leq k]}}^{-1}\leq C(N)\ R{\nu_0}^{-2}$.
	Integrating the resulting inequality over $t\in (-\tfrac{\nu_0}{2}\omega^{1-m}R^2,0)$ yields
	\begin{equation*}
		\begin{aligned}
			&\frac{\omega}{2^{n}}\abs{\Qrho\cap\left[\solA>\mu_-+\left(1-\tfrac{1}{2^{n+1}}\right)\omega\right]}\\
			&\leq C(N)\ \frac{R}{{\nu_0}^2} \norm{\nabla(\solA-(\mu_-+(1-\tfrac{1}{2^n})\omega))_+}_{L^2(\Qrho)}\abs{\Qrho\cap\left[k\leq\solA<l\right]}^{\frac{1}{2}}.
		\end{aligned}
	\end{equation*}
	We multiply the estimate by $\left(\frac{\omega}{2^{n}}\right)^{-1}$, use the inequality in \Cref{lem:Gradient.Estimate.via.Energy.Ineq} and square the resulting estimate to obtain
	\begin{equation*}
		\begin{aligned}
			&\abs{\Qrho\cap\left[\solA>\mu_-+\left(1-\tfrac{1}{2^{n+1}}\right)\omega\right]}^2\\
			&\quad\leq  C(N,c_1,c_2,m,\LipBoundSource,m_0)\ \nu_0^{-2}\abs{Q^{\nu_0}_\omega(R)}\\
			&\qquad\cdot\abs{\Qrho\cap\left[\mu_-+\left(1-\tfrac{1}{2^{n}}\right)\omega\leq\solA<\mu_-+\left(1-\tfrac{1}{2^{n+1}}\right)\omega\right]}\\
		\end{aligned}
	\end{equation*}
	Sum this estimate over $n=n_*,n_*+1,\ldots, n_0-1$ for some $n_0\in\IN$ and observe that
	\[
	\sum_{n=n_*}^{n_0-1}\abs{\Qrho\cap\left[\mu_-+(1-\tfrac{1}{2^{n}})\omega\leq\solA<\mu_-+(1-\tfrac{1}{2^{n+1}})\omega\right]}\leq \abs{\Qrho},
	\]
	so that
	\begin{equation*}
		(n_0-n_*)\abs{\Qrho\cap\left[\solA>\mu_-+\left(1-\tfrac{1}{2^{n_0}}\right)\omega\right]}^2\leq C(N,c_1,c_2,m,\LipBoundSource,m_0)\ \nu_0^{-4}\abs{\Qrho}^2.
	\end{equation*}
	Pick $n_0$ large enough such that
	\[
	\frac{C(N,c_1,c_2,m,\LipBoundSource,m_0)}{(n_0-n_*)\nu_0^4}\leq \nu^2
	\]
	to finish the proof.
\end{proof}
Now we can prove the second alternative of \Cref{lem:DeGiorgi-type} in the same manner as the first alternative, i.e.\ \Cref{prop:DeGiorgi.first.alternative}.
\begin{proposition}\label{prop:DiGiorgi.SecondAlternative}
	Let $\nu_0$ be given by \Cref{prop:DeGiorgi.first.alternative} and suppose \eqref{eq:DeGiorgi-type.alternative.II} holds.
	Then there exists $n_0\in\IN$ depending on $N$, $c_1$, $c_2$, $m$ and $\LipBoundSource$ such that
	\[
	\solA< \mu_-+(1-\tfrac{1}{2^{n_0}})\omega
	\]
	a.e.\ in $Q^{\nu_0}_\omega\left(\tfrac{R}{2}\right)$.
\end{proposition}
\begin{proof}
	Define the sequences
	\[
	R_n=\frac{R}{2}+\frac{R}{2^{n+1}},\quad k_n=\mu_-+(1-\frac{1}{2^{n_0}})\omega-\frac{\omega}{2^{n+1}},
	\]
	where $n_0$ is fixed and will be chosen later depending solely on $N$, $c_1$, $c_2$, $m$ and $\LipBoundSource$.
	Construct the family of nested and shrinking cylinders $Q_n=Q^{\nu_0}_\omega\left(R_n\right)$.
	Let $\zeta_n$ be the smooth parabolic cut-off function corresponding to the inclusion ${Q}_{n+1}\subseteq{Q}_{n}$ and note that
	\[
	\abs{\nabla\zeta_n}\leq \frac{2^{n+2}}{R},\quad 0\leq\zeta_t\leq C\ \omega^{m-1}\frac{2^{2n+2}}{\nu_0R^2}.
	\]
	Without loss of generality we may assume that 
	\[
	\mu_-+(1-\frac{1}{2^{n_0}})\omega \leq \mu_+.
	\]
	Otherwise \Cref{prop:DiGiorgi.SecondAlternative} holds trivially, since $\solA\leq\mu_+$ holds by definition.
	
	First, we observe that the following inequalities hold:
	\begin{gather*}
		\frac{\omega}{4}\leq k_n;\quad (\mu_+-k_n)^2= (\mu_+-\mu_--(1-\frac{1}{2^{n_0}}-\frac{1}{2^{n+1}})\omega)^2\leq (\mu_+-\mu_-)^2 \leq \omega^2;\\
		\mu_+\leq \frac{5}{4}\omega; \quad k_n^{-m_0}(\mu_+-k_n)\leq \left(\frac{\omega}{2}\right)^{-m_0} 2\omega.
	\end{gather*}
	We use \Cref{prop:interior.energy.estimate.upper.trunc} substituting $R$ by $R_n$, $k$ by $k_n$ and $\omega^{m-1}$ by $\tfrac{\nu_0}{2}\omega^{m-1}$ (recall that we set ${\bar{t}_0}=-\frac{\nu_0}{2}\omega^{1-m}\omega$) and we absorb all constants depending on $c_1$, $c_2$ and $m$ into $C$ to obtain
	\begin{equation*}
		\begin{aligned}
			&\norm{(\solA-k_n)_{+}\zeta_n}_{L^\infty({\bar{t}_0},{0};L^2(B_{R_{n+1}}))}^2+\omega^{m-1}\norm{\nabla(\solA-{k_n})_{+}\zeta_n}_{L^2({Q}_{n+1})}^2\\
			&\quad\leq C\ 2^{2n}\frac{\omega^{m+1}}{\nu_0R^2}\iint_{{Q}_{n}}\chi_{[\solA>k_n]}+C\ 2^{2n}\frac{\omega^{m+1}}{R^2}\iint_{{Q_n}}\chi_{[\solA>k_n]} \\
			&\qquad+C\ \LipBoundSource2^{-m_0}\frac{\omega^{m+1}}{R^2}\left(\frac{R^2}{\omega^{m+m_0}}\right)\iint_{Q_n}\chi_{[\solA>k_n]}\\
			&\quad\leq C(c_1,c_2,m,\LipBoundSource,m_0)\ 2^{2n}\frac{\omega^{m+1}}{R^2}\iint_{Q_n}\chi_{[\solA>k_n]},
		\end{aligned}
	\end{equation*}
	where we used that $\omega^{-m-m_0} R^2\leq R^{(m-m_0)/m}\leq 1$ by \eqref{eq:oscillation.Really.Large}.
	
	Introduce the change of variables
	\[
	\tau=\omega^{m-1}t,
	\]
	denote by $\tilde{Q}_n$ the transformed cylinder $B_{R_n}\times(-\frac{\nu_0}{2}R_n^2,0)$ and define the functions
	\[
	\solAvar(x,\tau)=\solA(x,\omega^{1-m}\tau),\quad\tilde{\zeta}_n(x,\tau)=\zeta_n(x,\omega^{1-m}\tau).
	\]
	The transformed estimate reads
	\begin{equation*}
		\begin{aligned}
			\norm{(\solAvar-k_n)_{+}\tilde{\zeta}_n}_{V^2(\tilde{Q}_{n+1})}^2\leq C(c_1,c_2,m,\LipBoundSource,m_0)\ 2^{2n}\frac{\omega^{2}}{R^2}\iint_{\tilde{Q}_{n}}\chi_{[\solAvar>k_n]}.
		\end{aligned}
	\end{equation*}
	We apply \Cref{lem:parab.space.embedding} to $(\solAvar-k_n)_{+}\tilde{\zeta}_n\in V^2_0$ and use the resulting estimate to conclude that
	\begin{align*}
		&\frac{\omega^2}{2^{2(n+1)}} |\tilde{Q}_{n+1}\cap[\solAvar>k_{n+1}] | = (k_{n+1}-k_n)^2|\tilde{Q}_{n+1}\cap[\solAvar>k_{n+1}] |\\
		&\quad\leq \norm{(\solAvar-k_n)_+\tilde{\zeta}_n}^2_{L^2(\tilde{Q}_n)}\leq C(N)\ |\tilde{Q}_{n}\cap[\solAvar>k_{n}] |^{\frac{2}{N+2}}\norm{(\solAvar-k_n)_+\tilde{\zeta}_n}^2_{V^2(\tilde{Q}_n)}\\
		&\quad\leq C(N,c_1,c_2,m,\LipBoundSource,m_0)\ 2^{2n}\frac{\omega^{2}}{R^2}\iint_{\tilde{Q}_{n}}\chi_{[\solAvar>k_n]}.
	\end{align*}
	
	Note that $|\tilde{Q}_n|= \abs{B_1}\frac{\nu_0}{2} R_n^{N+2}$, so
	\[
	\frac{|\tilde{Q}_n|^{1+\frac{2}{N+2}}}{|\tilde{Q}_{n+1}|}\leq C(N)\ \nu_0^{\frac{2}{N+2}}2^{N+2}R^2.
	\]
	Next, we set $y_n=\frac{\abs{\tilde{Q}_n\cap[\solAvar>k_n]}}{|\tilde{Q}_n|}$ and we conclude that
	\[
	y_{n+1}\leq C(N,c_1,c_2,m,\LipBoundSource,m_0)\ 2^{4n}y_n^{1+\frac{2}{N+2}}.
	\]
	By \Cref{lem:fast.geometric.convergence}, if
	\[
	y_0\leq C(N,c_1,c_2,m,\LipBoundSource,m_0)^{-\frac{N+2}{2}}(2^4)^{-\left(\frac{N+2}{2}\right)^2},
	\]
	then
	\[
	y_n\leq C(N,c_1,c_2,m,\LipBoundSource,m_0)^{-\frac{N+2}{2}}2^{-\left(\frac{N+2}{2}\right)^2} (2^4)^{-\frac{N+2}{2}n},
	\]
	so in this case $y_n\to 0$ as $n\to \infty$.
	
	Let us set
	\[
	\nu= C(N,c_1,c_2,m,\LipBoundSource,m_0)^{-\frac{N+2}{2}}(2^4)^{-\left(\frac{N+2}{2}\right)^2}
	\]
	and use \Cref{lem:alternative.condition.refined} to pick an $n_0$ corresponding to this $\nu$.
	Then we have that
	\[
	y_0=\frac{\abs{\tilde{Q}_0\cap[\solAvar>k_n]}}{|\tilde{Q}_0|}=\frac{\abs{Q^{\nu_0}_\omega(R)\cap[\solA>\mu_-+(1-\frac{1}{2^{n_0}})]}}{\abs{Q^{\nu_0}_\omega(R)}}\leq \nu
	\]
	by \Cref{lem:alternative.condition.refined}, so $y_n\to 0$ as $n\to \infty$.
	In particular, 
	\[
	\abs{Q_n\cap[\solA>\mu_-+(1-\frac{1}{2^{n_0}}-\frac{1}{2^{n+1}})\omega]}\to 0
	\]
	as $n\to\infty$, so $\abs{Q^{\nu_0}(\frac{R}{2})\cap[\solA\geq\mu_-+(1-\frac{1}{2^{n_0}})\omega]}=0$.
\end{proof}
The De Giorgi-type Lemma, that is, \Cref{lem:DeGiorgi-type}, follows by combining \Cref{prop:DeGiorgi.first.alternative,prop:DiGiorgi.SecondAlternative}.

\vspace*{0.2 cm}
\noindent\textbf{Declaration of competing interest.}
None.

\vspace*{0.2 cm}

\noindent\textbf{Data availability statement.}
My manuscript has no associated data.
	
\vspace*{0.2 cm}
	
\noindent\textbf{Acknowledgement.} 
The author thanks S.\ Sonner for introducing the topic and for the support and feedback during the writing process and N.\ Liao for the discussions on the regularity of Stefan problems.
	
\bibliographystyle{abbrv}
\bibliography{bibliography}
\end{document}